\documentclass[11pt, oneside]{amsart}   	% use "amsart" instead of "article" for AMSLaTeX format
\allowdisplaybreaks
\usepackage{geometry}                		% See geometry.pdf to learn the layout options. There are lots.
\usepackage{graphicx}				% Use pdf, png, jpg, or eps§ with pdflatex; use eps in DVI mode
\usepackage{amssymb}

\newtheorem{theorem}{\bf Theorem}[section]
\newtheorem{proposition}[theorem]{\bf Proposition}
\newtheorem{lemma}[theorem]{\bf Lemma}
\newtheorem{corollary}[theorem]{\bf Corollary}

\newtheorem{definition}[theorem]{\bf Definition}
\newtheorem{remark}[theorem]{\bf Remark}

\newenvironment{proofof}[1]{\noindent{\it Proof of
#1.}}{\hfill$\square$\\\mbox{}}

\title{Syzygies for the vector invariants of the dihedral group}

\author[M\'aty\'as Domokos]
{M\'aty\'as Domokos}
\address{Alfr\'ed R\'enyi Institute of Mathematics,
Re\'altanoda utca 13-15, 1053 Budapest, Hungary,
ORCID iD: https://orcid.org/0000-0002-0189-8831}
\email{domokos.matyas@renyi.hu}

\begin{document}

\thanks{Partially supported by the Hungarian National Research, Development and Innovation Office,  NKFIH K 138828,  K 132002.}


\maketitle

\begin{abstract}
The problem of finding generators of the $GL$-ideal of the relations between the generators of the algebra of invariants of the dihedral group acting on $m$-tuples of vectors from its defining $2$-dimensional representation is studied. It is shown that this $GL$-ideal is generated by relations depending on no more than $3$ vector variables. 
A minimal $GL$-ideal generating system is found for the case when $m=2$, 
and for the case of the dihedral group of order $8$ and arbitrary $m$. 
\end{abstract} 

\section{Introduction}\label{sec:intro} 

Given a group $G$ of linear transformations on a finite dimensional complex vector space $V$ and a positive integer $m$, consider the diagonal action of $G$ on the space $V^m$ of $m$-tuples of vectors from $V$. We call the algebra $R(m):=\mathbb{C}[V^m]^G$ of polynomial functions on $V^m$ constant along the $G$-orbits the \emph{algebra of vector invariants of} $G$. Following Weyl \cite{weyl} who gave a systematic study of the case when $G$ is one of the classical subgroups of the general linear group $GL(V)$, a description of the generators of $R(m)$ is referred to as a \emph{First Fundamental Theorem}, whereas a description of the generators of the ideal of relations between the generators of $R(m)$ 
 is referred to as a \emph{Second Fundamental Theorem} for the vector invariants of $G$. 
An indispensable tool to establish such theorems is to take into account 
a natural right action of the group $GL_m(\mathbb{C})$ of invertible $m\times m$ matrices on $V^m$ that commutes with the $G$-action: 
for $g=(g_{ij})_{i,j=1}^m \in GL_m(\mathbb{C})$ and $v=(v_1,\dots,v_m)\in V^m$ we set 
\[v\cdot g=(\sum_{i=1}^mg_{i1}v_i,\dots,\sum_{i=1}^mg_{im}v_i).\] 
This induces a left action of $GL_m(\mathbb{C})$ on $\mathbb{C}[V^m]$ via $\mathbb{C}$-algebra automorphisms. Namely, for $f\in \mathbb{C}[V^m]$, $v\in V^m$ and 
$g\in GL_m(\mathbb{C})$ we have $(g\cdot f)(v)=f(v\cdot g)$. 
The subalgebra $R(m)$ is a $GL_m(\mathbb{C})$-invariant subspace. Moreover, $R(m)$ is a graded subalgebra of $\mathbb{C}[V^m]$, where the latter is endowed with the standard grading. Write $R(m)_+$ for the maximal ideal of $R(m)$ spanned  by its homogeneous elements of positive degree. It is clearly a $GL_m(\mathbb{C})$-submodule, just like $(R(m)_+)^2$.  
Let $W(m)$ be a $GL_m(\mathbb{C})$-module direct complement of $(R(m)_+)^2$ in $R(m)_+$. 
Then the subspace $W(m)$ minimally generates the algebra $R(m)$. 
For example, for the case when $G=GL_3(\mathbb{C})$ acting on the space $V=\mathbb{C}^{3\times 3}$ of $3\times 3$ matrices by conjugation, an explicit description of a minimal homogeneous generating system of $R(m)$ for general $m$ is rather complicated. This was made transparent in \cite{abeasis-pittaluga} by determining 
the $GL_m(\mathbb{C})$-module structure of $W(m)$.   
Let us turn next to the relations between the generators of $R(m)$. 
The identity map $W(m)\to W(m)$ induces a $\mathbb{C}$-algebra surjection 
\begin{equation}\label{eq:varphi(m)} \varphi(m):S(W(m))\to R(m)\end{equation}  
from the symmetric tensor algebra $S(W(m))$ of $W(m)$ onto $R(m)$, and 
$\ker(\varphi(m))$ is the \emph{ideal of relations for the minimal generating subspace} $W(m)$ of $R(m)$. Clearly, $\ker(\varphi(m))$ is a \emph{$GL$-ideal}, i.e.  a $GL_m(\mathbb{C})$-stable ideal in $S(W(m))$. A sensible way to describe $\ker(\varphi(m))$ is to find elements in it which 
generate irreducible $GL_m(\mathbb{C})$-submodules of $S(W(m))$, such that the sum of these $GL_m(\mathbb{C})$-submodules is direct and forms a minimal generating subspace of the ideal $\ker(\varphi(m))$ (so in particular,  these elements  constitute a minimal generating system of $\ker(\varphi(m))$ as a $GL$-ideal).  For example,  in the special case $G=GL_3(\mathbb{C})$ and $V=\mathbb{C}^{3\times 3}$, the $GL_m(\mathbb{C})$-module structure of the minimal degree non-zero homogeneous component of $\ker(\varphi(m))$ was computed in \cite{benanti-drensky}. 

In this paper we study the problem of finding generators of the $GL$-ideal $\ker(\varphi(m))$ for the 
defining $2$-dimensional representation of the dihedral group $D_{2n}$ of order $2n$  ($n\ge 3$ is a positive integer). 
For the case $n=3$ the result  can be easily deduced from 
\cite{domokos-puskas}, see Theorem~\ref{thm:n=3}. For arbitrary $n$ and  $m=2$ the solution will be given in 
Theorem~\ref{thm:m=2}. Moreover, for arbitrary $n$ we reduce the problem to the case $m=3$ (see Theorem~\ref{thm:m=3 sufficient}).  
As an application of this reduction we give a complete solution of the problem when 
 $n=4$ in Theorem~\ref{thm:n=4}.

\section{First Fundamental Theorem for the dihedral group}\label{sec:fft for dihedral} 

Let $n\ge 3$ be a positive integer and $\omega$ a complex primitive $n$th root of $1$. 
We take for $G$ the dihedral group $D_{2n}$, the subgroup of  $GL_2(\mathbb{C})$ generated by the matrices 
$\left(\begin{array}{cc}\omega & 0 \\0 & \omega^{-1}\end{array}\right)$ and 
$\left(\begin{array}{cc} 0 & 1 \\ 1 & 0 \end{array}\right)$.  The given representation of $D_{2n}$ is the complexification of the defining real representation of $D_{2n}$ as 
the group of isometries of the Euclidean plane mapping a given regular $n$-gon into itself.  
Denoting by $x,y$ the coordinate functions on $V:=\mathbb{C}^2$, the algebra $\mathbb{C}[x,y]^{D_{2n}}$ of polynomial invariants is generated by the algebraically independent invariants 
\[q:=xy,\qquad p:=x^n+y^n.\] 
Consider now the diagonal action of $D_{2n}$ on the space $V^m:=V\oplus\cdots\oplus V$ ($m$ direct summands) of $m$-tuples of vectors. Denote by $x_i$ (respectively $y_i$) the function mapping an $m$-tuple of vectors to the first 
(respectively second) coordinate of its $i$th vector component, $i=1,\dots,m$. The coordinate ring $\mathbb{C}[V^m]$ is the $2m$-variable polynomial ring $\mathbb{C}[x_1,y_1,\dots,x_m,y_m]$. 
For a degree $d$ homogeneous element $g\in \mathbb{C}[x,y]$ and $\alpha=(\alpha_1,\dots,\alpha_m)$ with $\sum_j\alpha_j=d$ denote by $g_{\alpha}$ the multihomogeneous component of $\frac{1}{\binom{d}{\alpha}}g(x_1+\cdots+x_m,y_1+\cdots+y_m)$ having total degree $\alpha_j$ in $x_j,y_j$ for $j=1,\dots,m$; so 
\[g(x_1+\cdots+x_m,y_1+\cdots+y_m)=\sum_{\alpha}\binom{d}{\alpha}g_{\alpha}, \] 
where $\binom{d}{\alpha}=\frac{d!}{\alpha_1!\cdots \alpha_m!}$. 
The elements $g_{\alpha}\in \mathbb{C}[V^m]$ are called the \emph{polarizations} of $g$, and they form a $\mathbb{C}$-vector space basis in the $GL_m(\mathbb{C})$-submodule 
$\langle g\rangle_{GL_m(\mathbb{C})}$ of $\mathbb{C}[V^m]$ generated by $g$. 
Here we treat $\mathbb{C}[V]$ as the subalgebra of $\mathbb{C}[V^m]$ consisting of the polynomial functions that depend only on the first vector component of the $m$-tuples in $V^m$, and throughout the paper we write $\langle A\rangle_{GL_m(\mathbb{C})}$ for 
the $GL_m(\mathbb{C})$-submodule generated by the element or subset $A$ of a given $GL_m(\mathbb{C})$-module. 
 
For example, the polarizations of $p\in \mathbb{C}[x,y]^{D_{2n}}$ are 
\[p_{\alpha}=x_1^{\alpha_1}\cdots x_m^{\alpha_m}+y_1^{\alpha_1}\cdots y_m^{\alpha_m} 
\text{ for }\alpha\in\mathbb{N}_0^m \text{ with }\sum \alpha_i=n\] 
and the polarizations of $q\in \mathbb{C}[x,y]^{D_{2n}}$ for $m=2$ are 
\[q_{2,0}=x_1y_1,\quad q_{11}=\frac 12(x_1y_2+y_1x_2),\quad q_{0,2}=x_2y_2. \] 

Our starting point is the following known description of the generators of 
$\mathbb{C}[V^m]^{D_{2n}}$ due to Hunziker: 

\begin{theorem}\label{thm:hunziker} \cite[Theorem 4.1]{hunziker} The polarizations of $q$ and $p$ form a minimal homogeneous generating system of the  $\mathbb{C}$-algebra $\mathbb{C}[V^m]^{D_{2n}}$. 
\end{theorem} 

\begin{remark} {\rm In fact \cite{hunziker} works over the real field, but by well known basic principles \cite[Theorem 4.1]{hunziker} implies the variant Theorem~\ref{thm:hunziker} above. }
\end{remark}  

Therefore in this case $W(m)$ (cf. Section~\ref{sec:intro}) is 
\[W(m)=\langle q\rangle_{GL_m(\mathbb{C})}\oplus \langle p\rangle_{GL_m(\mathbb{C})}
=\mathrm{Span}_{\mathbb{C}}\displaystyle\{q_{\alpha},p_{\beta}\mid \alpha,\beta\in \mathbb{N}_0^m,\ \sum_{i=1}^m\alpha_i=2,\ \sum_{j=1}^m\beta_j=n\}. 
\] 

We have the $GL_m(\mathbb{C})$-module isomorphisms 
\[\langle q\rangle_{GL_m(\mathbb{C})}\cong S^2(\mathbb{C}^m) \quad \text{ and }\quad \langle p\rangle_{GL_m(\mathbb{C})}\cong 
 S^n(\mathbb{C}^m)\] 
where $S^k(U)$ stands for the $k$th symmetric tensor power of the $GL(U)$-module $U$. 

Write $\mathcal{F}(n,m):=S(W(m))=\sum_{d=0}^\infty S^d(W(m))$ for the symmetric tensor algebra of $W(m)$ (endowed with the natural $GL_m(\mathbb{C})$-module structure induced by the $GL_m(\mathbb{C})$-action on $W(m)$).  
Although $S^1(W(m))\subset \mathcal{F}(n,m)$ is just $W(m)$, to avoid confusion later we need to distinguish in the notation the elements of $W(m)$ when they are considered as elements of $\mathcal{F}(n,m)$. We shall write $\rho_\alpha\in \mathcal{F}(n,m)$ for the element corresponding to $q_\alpha$, and $\pi_{\beta}\in \mathcal{F}(n,m)$ for the element corresponding to $p_\beta$. So $\mathcal{F}(n,m)$ is the polynomial ring 
\[\mathcal{F}(n,m)=\mathbb{C}[\rho_\alpha,\pi_\beta\mid \sum\alpha_i=2,\ \sum\beta_j=n].\] 
Thus in our case  $\varphi(m)$ (cf. \eqref{eq:varphi(m)} in Section~\ref{sec:intro}) is the $\mathbb{C}$-algebra (as well as $GL_m(\mathbb{C})$-module)   surjection 
\[\varphi(n,m):\mathcal{F}(n,m)\to \mathbb{C}[V^m]^{D_{2n}},\qquad \rho_\alpha\mapsto q_\alpha,\ \pi_\beta\mapsto p_\beta.\] 
Throughout the paper for $l\le m$ we shall treat $\mathbb{C}[V^l]$ and 
$\mathcal{F}(n,l)$ as a subalgebra of $\mathbb{C}[V^m]$ and $\mathcal{F}(n,m)$ in the obvious way. 
Our aim is to describe the kernel $\ker(\varphi(n,m))$ of $\varphi(n,m)$ as a $GL$-ideal. 

\section{Preliminaries on $GL_m(\mathbb{C})$-modules} \label{sec:Lie action}

As a general reference to the material in this section, see for example the book 
\cite{procesi}. 
Recall that the isomorphism classes of the irreducible polynomial $GL_m(\mathbb{C})$-modules of degree $d$ are labeled by partitions of $d$ with  $m$ parts, where by a partition $\lambda$ of $d$ with $m$ parts (notation: $\lambda\in\mathrm{Par}_m(d)$) we mean a sequence $(\lambda_1,\dots,\lambda_m)$ of non-negative integers with 
$\lambda_1\ge\cdots\ge\lambda_m$ and $\lambda_1+\cdots+\lambda_m=d$. 
We denote by 
$\mathrm{ht}(\lambda)$ the number of non-zero elements in the sequence 
$(\lambda_1,\dots,\lambda_m)$. Moreover, for $l<m$ and non-negative integers 
$\lambda_1\ge\cdots\ge\lambda_l$ with $\sum\lambda_i=d$ we identify the sequence $(\lambda_1,\dots,\lambda_l)$ with $(\lambda_1,\dots,\lambda_l,0,\dots,0)\in \mathrm{Par}_m(d)$. 

For an $\ell$-dimensional vector space $U$ write $S^{\lambda}(U)$  for the $GL(U)$-module associated to $U$ by the Schur functor $S^{\lambda}(-)$. This is an irreducible polynomial $GL(U)$-module if 
$\mathrm{ht}(\lambda)\le \ell$ and $S^{\lambda}(U)$ is the zero module when 
$\mathrm{ht}(\lambda)>\ell$. 
For example, for the partition $\lambda=(n,0,\dots)$ with only one non-zero part we have 
$S^{\lambda}(U)=S^n(U)$, the $n$th symmetric tensor power of $U$. 
Moreover, for $\lambda\in\mathrm{Par}_m(d)$, $S^{\lambda}(\mathbb{C}^m)$ is the irreducible polynomial $GL_m(\mathbb{C})$-module labeled by $\lambda$. 
Note that when $U$ has a $GL_m(\mathbb{C})$-module structure ($m$ may differ from $\ell$), then $S^{\lambda}(U)$ becomes naturally a (typically not irreducible) $GL_m(\mathbb{C})$-module.  

$\mathcal{F}(n,m)$ and $\mathbb{C}[V^m]$ are polynomial $GL_m(\mathbb{C})$-modules. 
An element $v$ of a polynomial $GL_m(\mathbb{C})$-module generates a submodule 
isomorphic to $S^{\lambda}(\mathbb{C}^m)$ if 
\begin{enumerate} 
\item it is fixed by $UT_m(\mathbb{C})$, the subgroup of upper triangular unipotent matrices; 
\item for a diagonal element 
$\mathrm{diag}(z_1,\dots,z_m)\in GL_m(\mathbb{C})$ we have 
$\mathrm{diag}(z_1,\dots,z_m)\cdot v=z_1^{\lambda_1}\cdots z_m^{\lambda_m}v$.  
\end{enumerate} 
Such an element is called a \emph{highest weight vector}. An irreducible polynomial $GL_m(\mathbb{C})$-module contains a unique (up to non-zero scalar multiples) highest weight vector. 

The action of $GL_m(\mathbb{C})$ on $\mathcal{F}(n,m)$ and on $\mathbb{C}[V^m]$ induces a representation of its Lie algebra $\mathfrak{gl}_m(\mathbb{C})$ on 
$\mathcal{F}(n,m)$ and on  $\mathbb{C}[V^m]$, such that $\varphi(n,m)$ is a homomorphism of 
 $\mathfrak{gl}_m(\mathbb{C})$-modules. 
In particular, $\ker(\varphi(n,m))$ is preserved by  $\mathfrak{gl}_m(\mathbb{C})$. 
To detect explicit highest weight vectors it is convenient to pass to the Lie algebra action. 
Therefore 
for later use we shall record the formulae determining this 
$\mathfrak{gl}_m(\mathbb{C})$-representation on $\mathcal{F}(n,m)$ in the case $m=3$.  
Denote by $E_{i,j}$ the matrix unit having entry $1$ in the $(i,j)$ position and the entry $0$ in all other positions. The Lie algebra $\mathfrak{gl}_m(\mathbb{C})$ of $GL_m(\mathbb{C})$ has basis $\{E_{i,j}\mid 1\le i,j\le m\}$. We have 
\begin{align*} E_{1,2}.\pi_{i,j,k}=j\pi_{i+1,j-1,k}, \quad E_{2,3}.\pi_{i,j,k}=k\pi_{i,j+1,k-1}, \quad E_{ss}.\pi_{\alpha_1,\alpha_2,\alpha_3}=\alpha_s \pi_{\alpha}
 \\ E_{2,1}.\pi_{i,j,k}=i\pi_{i-1,j+1,k},\qquad E_{3,2}.\pi_{i,j,k}=j\pi_{i,j-1,k+1} 
 \\ E_{1,2}.\rho_{i,j,k}=j\rho_{i+1,j-1,k}, \quad E_{2,3}.\rho_{i,j,k}=k\rho_{i,j+1,k-1}, \quad E_{ss}.\rho_{\alpha_1,\alpha_2,\alpha_3}=\alpha_s \rho_{\alpha}
 \\ E_{2,1}.\rho_{i,j,k}=i\rho_{i-1,j+1,k},\qquad E_{3,2}.\rho_{i,j,k}=j\rho_{i,j-1,k+1} 
 \end{align*}  
where in the above formulae $\pi_{a,b,c}$ or $\rho_{a,b,c}$ is interpreted as zero unless $a,b,c$ are all non-negative.

\section{Reduction to the case $m=3$} \label{sec:reduction to m=3}

The algebra $\mathbb{C}[V^m]^{D_{2n}}$ is a graded subalgebra of $\mathbb{C}[V^m]$, where the latter is  endowed with the standard grading. We introduce a grading on $\mathcal{F}(n,m)$ by making the homomorphism $\varphi(n,m)$ degree preserving; that is, the degree of 
the variables $\rho_\alpha$ is $2$, and the degree of the variables $\pi_\beta$ is $n$. 
We shall denote by $\mathcal{F}(n,m)_d$ and by  $\ker(\varphi(n,m))_d$ the degree $d$ homogeneous component of $\mathcal{F}(n,m)$ and its graded subspace $\ker(\varphi(n,m))$, and we write $\mathcal{F}(n,m)_{\le d}$ and $\ker(\varphi(n,m))_{\le d}$ 
for the sum of the homogeneous components of degree at most $d$ in $\mathcal{F}(n,m)$ and $\ker(\varphi(n,m))$.  

\begin{proposition}\label{prop:2n+2} 
The ideal $\ker(\varphi(n,m))$ is generated by its elements of degree at most $2n+2$. 
\end{proposition} 

\begin{proof}
By a general result of Derksen \cite[Theorem 2]{derksen} the ideal $\ker(\varphi(n,m))$ is generated by its elements of degree at most $2\tau_{D_{2n}}(V^m)$, where $\tau_{D_{2n}}(V^m)$ denotes the minimal positive integer $k$ such that all homogeneous polynomials in $\mathbb{C}[V^m]$ of degree $k$ are contained in the ideal 
$\mathbb{C}[V^m]^{D_{2n}}_+\mathbb{C}[V^m]$ of $\mathbb{C}[V^m]$ generated by the homogeneous $D_{2n}$-invariants of positive degree (called the Hilbert ideal). 
The number $\tau_{D_{2n}}(V^m)$ is bounded by the Noether number of $D_{2n}$  by \cite[Lemma 1.6]{cziszter-domokos_lowerbound}. 
The Noether number of $D_{2n}$ is $n+1$ by \cite{schmid}  (see also  \cite[Corollary 5.6]{cziszter-domokos_indextwo} for a stronger statement). This implies 
$\tau_{D_{2n}}(V^m)\le n+1$. 
Note finally that  obviously $\tau_{D_{2n}}(V^m)\ge\tau_{D_{2n}}(V)$, and the latter 
number is known to be $n+1$: 
indeed, $\mathbb{C}[V]^{D_{2n}}$ is generated by algebraically independent homogeneous 
invariants of degree $2$ and $n$, and thus  the Hilbert series of the corresponding coinvariant algebra 
$\mathbb{C}[V^m]/\mathbb{C}[V^m]^{D_{2n}}_+\mathbb{C}[V^m]$ equals 
$(1+t)(1+t+\cdots+t^{n-1})$ (see \cite{chevalley}). 
\end{proof} 
 
Note that $\mathcal{F}(n,m)$ has the tensor product decomposition 
$\mathcal{F}(n,m)=\mathcal{D}(m)\otimes \mathcal{E}(n,m)$
where 
\[\mathcal{D}(m)=S(\langle q\rangle_{GL_m(\mathbb{C})})=\mathbb{C}[\rho_{\alpha}\mid \alpha\in \mathbb{N}_0^m,\ \sum\alpha_i=2]\] 
and 
\[\mathcal{E}(n,m)=S(\langle p \rangle_{GL_m(\mathbb{C})})=\mathbb{C}[\pi_{\beta}\mid \beta\in \mathbb{N}_0^m,\ \sum\beta_j=n]
\] 
 For $m\ge 3$ set 
 \[\mathcal{R}_{2,2,2}:=\det \left(\begin{array}{ccc}\rho_{2,0,0} & \rho_{1,1,0} & \rho_{1,0,1} \\ \rho_{1,1,0} & \rho_{0,2,0} & \rho_{0,1,1} \\ \rho_{1,0,1} & \rho_{0,1,1} & \rho_{0,0,2}\end{array}\right)\in \mathcal{D}(m)\subset \mathcal{F}(n,m).\] 
 
 \begin{lemma}\label{lemma:R_2,2,2}
For $m\ge 3$  the element $\mathcal{R}_{2,2,2}$ 
belongs to $\ker(\varphi(n,m))$, 
and $\langle\mathcal{R}_{2,2,2}\rangle_{GL_m(\mathbb{C})}\cong S^{(2,2,2)}(\mathbb{C}^m)$ as $GL_m(\mathbb{C})$-modules. 
\end{lemma} 

\begin{proof} The matrix 
 \[\left(\begin{array}{ccc}q_{2,0,0} & q_{1,1,0} & q_{1,0,1} \\ q_{1,1,0} & q_{0,2,0} & q_{0,1,1} \\q_{1,0,1} & q_{0,1,1} & q_{0,0,2}\end{array}\right)=
 \left(\begin{array}{cc} x_1 & y_1 \\x_2 & y_2 \\x_3 & y_3\end{array}\right)\cdot 
 \left(\begin{array}{cc} 0 & 1/2 \\ 1/2 & 0 \end{array}\right)\cdot 
\left(\begin{array}{ccc} x_1 & x_2 & x_3 \\ y_1 & y_2 & y_3\end{array}\right)\] 
has rank $2$, implying that its determinant is zero. On the other hand, this determinant is 
$\varphi(n,m)(\mathcal{R}_{2,2,2})$.  Thus $\mathcal{R}_{2,2,2}\in \ker(\varphi(n,m))$. 
A straightforward calculation yields $E_{1,2}.\mathcal{R}_{2,2,2}=0$, 
$E_{2,3}.\mathcal{R}_{2,2,2}=0$, and $E_{i,i}.\mathcal{R}_{2,2,2}=2\mathcal{R}_{2,2,2}$ 
($i=1,2,3)$, hence $\mathcal{R}_{2,2,2}$ is a highest weigh vector in $\mathcal{F}(n,m)$ with weight $(2,2,2)$. 
\end{proof} 

\begin{remark}
In fact it is known that $\mathcal{R}_{2,2,2}$ generates as a $GL$-ideal the kernel of the restriction of $\varphi(n,m)$ to $\mathcal{D}(m)$ (see the Second Fundamental Theorem for the orthogonal group in \cite[Theorem 2.17.A]{weyl}). 
\end{remark} 

Setting $\bar{\mathcal{D}}(m):=\mathcal{D}(m)/\langle \mathcal{R}_{2,2,2}\rangle_{GL_m(\mathbb{C})}\mathcal{D}(m)$, by 
Lemma~\ref{lemma:R_2,2,2} we conclude that $\varphi(n,m)$ factors through 
the natural surjection $\mathcal{F}(n,m)\to \bar{\mathcal{D}}(m)\otimes \mathcal{E}(n,m)$, so 
we get the graded $GL_m(\mathbb{C})$-module algebra surjection  
\[\bar\varphi(n,m): \bar{\mathcal{D}}(m)\otimes \mathcal{E}(n,m)
\to \mathbb{C}[V^m]^{D_{2n}}.\] 

\begin{lemma}\label{lemma:obvious} 
For $m\ge 3$ the $GL$-ideal $\ker(\varphi(n,m))$ is generated by $\mathcal{R}_{2,2,2}$ and 
any subset of $\ker(\varphi(n,m))$ whose image under the natural surjection $\mathcal{F}(n,m)\to \bar{\mathcal{D}}(m)\otimes \mathcal{E}(n,m)$ generates $\ker(\bar\varphi(n,m))$ as a $GL$-ideal.  
\end{lemma}

\begin{proof}
This is an immediate consequence of the construction of $\bar\varphi(n,m)$. 
\end{proof} 

Our next aim is to compute the $GL_m(\mathbb{C})$-module structure of the homogeneous components of $\bar{\mathcal{D}}(m)\otimes \mathcal{E}(n,m)$ up to  degree $2n+2$. 

\begin{proposition}\label{prop:well known GL-module isomorphisms}
We have the following isomorphisms of $GL_m(\mathbb{C})$-modules: 
\begin{itemize}
\item[(i)] $\mathcal{D}(m)\cong \sum_{d=0}^{\infty}\sum_{\lambda\in\mathrm{Par}_m(d)}S^{2\lambda}(\mathbb{C}^m)$
where $2\lambda$ stands for the partition $(2\lambda_1,\dots,2\lambda_m)$ of $2d$. 
\item[(ii)] 
$\bar{\mathcal{D}}(m)\cong  \sum_{d=0}^{\infty}\sum_{\lambda\in\mathrm{Par}_m(d),\ \mathrm{ht}(\lambda)\le 2}S^{2\lambda}(\mathbb{C}^m)$. 
\item[(iii)] $\mathcal{E}(n,m)\cong S(S^n(\mathbb{C}^m))
=\bigoplus_{d=0}^\infty S^d(S^n(\mathbb{C}^m))$ (the symmetric tensor algebra of 
$S^n(\mathbb{C}^m)$. 
\item[(iv)] $S^2(S^n(\mathbb{C}^m))\cong 
\sum_{j=0}^{\lfloor \frac n2\rfloor} S^{(2n-2j,2j)}(\mathbb{C}^m)$.  
\end{itemize}
\end{proposition} 

\begin{proof} (i) We have $\langle q\rangle_{GL_m(\mathbb{C})}\cong S^2(\mathbb{C}^m)$ and therefore (i) follows from the well-known decomposition 
\[S^d(S^2(\mathbb{C}^m))\cong \sum_{\lambda\in\mathrm{Par}_m(d)} S^{2\lambda}(\mathbb{C}^m)\] 
(see for example \cite[Section 11.4.5, Theorem]{procesi}). 

(ii) $\bar{\mathcal{D}}(m)$ can be identified with the coordinate ring of the variety of 
$m\times m$ symmetric matrices of rank at most $2$ endowed with the natural $GL_m(\mathbb{C})$-action. For the well-known decomposition of this $GL_m(\mathbb{C})$-module see for example \cite[Section 11.5.1, Second Fundamental Theorem]{procesi}. 

(iii) follows from the isomorphism 
$\langle p\rangle_{GL_m(\mathbb{C})}\cong S^n(\mathbb{C}^m)$. 

(iv) can be found for example in \cite[Example I.8.9]{macdonald} 
(see also Remark~\ref{remark:elliott} and the proof of Theorem~\ref{thm:m=2} (i)). 
\end{proof} 

Write $\bar{\mathcal{D}}(m)_d$, $\mathcal{E}(n,m)_d$ for the degree $d$ 
homogeneous component of  $\bar{\mathcal{D}}(m)$, $\mathcal{E}(n,m)$, and write 
$(\bar{\mathcal{D}}(m)\otimes \mathcal{E}(n,m))_{\le d}$ for the sum of the homogeneous components of $\bar{\mathcal{D}}(m)\otimes \mathcal{E}(n,m)$ of degree at most $d$. 

\begin{corollary} \label{cor:GL-structure of F} 
As a $GL_m(\mathbb{C})$-module, 
$(\bar{\mathcal{D}}(m)\otimes \mathcal{E}(n,m))_{\le 2n+2}$ is isomorphic to 
\begin{align*} 
\sum_{d=0}^{n+1}\sum_{\lambda\in\mathrm{Par}_m(d),\ \mathrm{ht}( \lambda)\le 2} S^{2\lambda}(\mathbb{C}^m)
+\sum_{d=0}^{\lfloor \frac{n+2}{2}\rfloor}\sum_{\lambda\in\mathrm{Par}_m(d),\ \mathrm{ht}(\lambda)\le 2}
S^n(\mathbb{C}^m)\otimes S^{2\lambda}(\mathbb{C}^m)
\\ +\sum_{j=0}^{\lfloor\frac n2\rfloor}(S^{(2n-2j,2j)}(\mathbb{C}^m)+S^{(2n-2j,2j)}(\mathbb{C}^m)\otimes S^2(\mathbb{C}^m)). 
\end{align*} 
\end{corollary} 

\begin{proof} 
As $\mathcal{E}(n,m)_d$ is non-zero only if $d$ is a multiple of $n$, 
the sum of the homogeneous components of $\bar{\mathcal{D}}(m)\otimes \mathcal{E}(n,m)$ of degree at most $2n+2$ is 
\[\sum_{d=0}^{2n+2}\bar{\mathcal{D}}(m)_d\otimes \mathcal{E}(n,m)_0
+\sum_{d=0}^{n+2}\bar{\mathcal{D}}(m)_d\otimes \mathcal{E}(n,m)_n
+\sum_{d=0}^2\bar{\mathcal{D}}(m)_d\otimes \mathcal{E}(n,m)_{2n}.\] 
Now the statement follows from the $GL_m(\mathbb{C})$-module isomorphisms given in 
Proposition~\ref{prop:well known GL-module isomorphisms}. 
\end{proof}

Recall that for $l\le m$ we view $\mathcal{F}(n,l)$ as a subspace of $\mathcal{F}(n,m)$ in the obvious way. 
Similarly, $\mathcal{D}(l)$, $\bar{\mathcal{D}}(l)$, $\mathcal{E}(n,l)$, are viewed as subspaces of $\mathcal{D}(m)$, $\bar{\mathcal{D}}(m)$, $\mathcal{E}(n,m)$, and so 
$\bar{\mathcal{D}}(l)\otimes \mathcal{E}(n,l)$ is viewed as a subspace 
of $\bar{\mathcal{D}}(m)\otimes\mathcal{E}(n,m)$.

\begin{proposition}\label{prop:height 3} 
Let $U$ be a  $GL_m(\mathbb{C})$-invariant subspace of $(\bar{\mathcal{D}}(m)\otimes\mathcal{E}(n,m))_{\le 2n+2}$. 
Then 
\[U=\langle U\cap \bar{\mathcal{D}}(3)\otimes\mathcal{E}(n,3)\rangle_{GL_m(\mathbb{C})}.\] 
\end{proposition} 

\begin{proof} By complete reducibility of $GL_m(\mathbb{C})$-modules, it is sufficient to prove the statement in the special case when $U$ is a \emph{minimal}  $GL_m(\mathbb{C})$-invariant subspace of $(\bar{\mathcal{D}}(m)\otimes\mathcal{E}(n,m))_{\le 2n+2}$. So  assume that $U$ is a minimal $GL_m(\mathbb{C})$-invariant subspace. 
Pieri's Formula (cf. \cite[I.5.16]{macdonald}) implies that if $S^{\nu}(\mathbb{C}^m)$ occurs as a summand in $S^{\lambda}(\mathbb{C}^m)\otimes S^{\mu}(\mathbb{C}^m)$ then 
$\mathrm{ht}(\nu)\le \mathrm{ht}(\lambda)+\mathrm{ht}(\mu)$. 
Thus by Corollary~\ref{cor:GL-structure of F} we conclude that  
$U\cong  S^{\lambda}(\mathbb{C}^m)$ for some 
partition $\lambda$ with $\mathrm{ht}(\lambda)\le 3$.  
It follows that $U=\langle w\rangle_{GL_m(\mathbb{C})}$ for some highest weight vector 
$w$ with weight $\lambda$. That is, $w$ is a non-zero element of $U$ fixed by the subgroup of unipotent upper triangular matrices in $GL_m(\mathbb{C})$, and for a diagonal element 
$\mathrm{diag}(z_1,\dots,z_m)\in GL_m(\mathbb{C})$ we have $\mathrm{diag}(z_1,\dots,z_m)\cdot w=
z_1^{\lambda_1}\cdots z_m^{\lambda_m}w$. In particular, $\mathrm{ht}(\lambda)\le 3$ implies that $w$ belongs to 
$\bar{\mathcal{D}}(3)\otimes\mathcal{E}(n,3)$. 
\end{proof} 

We arrived at the main result of this Section: 

\begin{theorem}\label{thm:m=3 sufficient} 
For $m\ge 3$ the ideal $\ker(\varphi(n,m))$ of $\mathcal{F}(n,m)$ 
is generated  by the 
$GL_m(\mathbb{C})$-submodule 
$\langle \ker(\varphi(n,3))_{\le 2n+2}\rangle_{GL_m(\mathbb{C})}$
of $\mathcal{F}(n,m)$. 
That is, $\ker(\varphi(n,m))$ is generated as a $GL$-ideal 
by elements of degree $\le 2n+2$ contained in the kernel of $\varphi(n,3)$. 
\end{theorem} 

\begin{proof} 
The natural surjection $\mathcal{F}(n,m)\to \bar{\mathcal{D}}(m)\otimes \mathcal{E}(n,m)$ maps 
$\ker(\varphi(n,m))$ onto $\ker(\bar\varphi(n,m))$, therefore by Proposition~\ref{prop:2n+2} the latter is generated as an ideal by $\ker(\bar\varphi(n,m))_{\le 2n+2}$. 
Now $\ker(\bar\varphi(n,m))_{\le 2n+2}$ is a $GL_m(\mathbb{C})$-invariant subspace 
of $(\bar{\mathcal{D}}(m)\otimes\mathcal{E}(n,m))_{\le 2n+2}$, hence by Proposition~\ref{prop:height 3} we have 
\[\ker(\bar\varphi(n,m))_{\le 2n+2}=\langle \ker(\bar\varphi(n,m))_{\le 2n+2}\cap 
 \bar{\mathcal{D}}(3)\otimes\mathcal{E}(n,3)\rangle_{GL_m(\mathbb{C})}
 =\langle \ker(\bar\varphi(n,3))_{\le 2n+2}\rangle_{GL_m(\mathbb{C})}.\] 
So the image $\ker(\bar\varphi(n,3))_{\le 2n+2}$ of $\ker(\varphi(n,3))_{\le 2n+2}$ under the natural surjection 
$\mathcal{F}(n,m)\to \bar{\mathcal{D}}(m)\otimes \mathcal{E}(n,m)$ generates $\ker(\bar\varphi(n,m))$ as a $GL$-ideal. Since the element $\mathcal{R}_{2,2,2}$ belongs to  $\ker(\varphi(n,3))_{\le 2n+2}$,  
we can conclude by Lemma~\ref{lemma:obvious} 
that $\ker(\varphi(n,3))_{\le 2n+2}$ generates $\ker(\varphi(n,m))$ as a $GL$-ideal. 
Equivalently, $\ker(\varphi(n,m))$ is generated as an ideal of 
$\mathcal{F}(n,m)$ by its subspace 
$\langle \ker(\varphi(n,3))_{\le 2n+2}\rangle_{GL_m(\mathbb{C})}$. 
\end{proof}

\section{The case $m=2$} \label{sec:m=2} 

Set 
\begin{equation}\label{eq:R(n,2)}
\mathcal{R}(n)_{n,2}:=\pi_{n,0}\rho_{0,2}-2\pi_{n-1,1}\rho_{1,1}+\pi_{n-2,2}\rho_{2,0}.\end{equation} 

For $k=1,\dots,\lfloor \frac n2\rfloor$ set 
\begin{align}\label{eq:R(2n-2k,2k)} 
\mathcal{R}(n)_{2n-2k,2k}:=
(-1)^k\frac 12\binom{2k}{k}\pi_{n-k,k}^2+\sum_{j=0}^{k-1}(-1)^j\binom{2k}{j}\pi_{n-j,j}\pi_{n-2k+j,2k-j}
\\ \notag -4^k\rho_{2,0}^{n-2k}(\rho_{1,1}^2-\rho_{2,0}\rho_{0,2})^k.  
\end{align} 

\begin{theorem}\label{thm:m=2} 
For $\lambda\in\{(n,2),(2n-2,2),(2n-4,4),\dots,(2n-2\lfloor\frac n2\rfloor,
2\lfloor\frac n2\rfloor)\}$ consider the element $\mathcal{R}(n)_{\lambda}$  in $\mathcal{F}(n,2)$ introduced above. 
\begin{itemize} 
\item[(i)]For $m\ge2$ the element $\mathcal{R}(n)_{\lambda}$ generates a $GL_m(\mathbb{C})$-submodule in $\ker(\varphi(n,m))$ isomorphic to $S^{\lambda}(\mathbb{C}^m)$. 
\item[(ii)] For $m=2$ the ideal, $\ker(\varphi(n,2))$ is minimally generated by the above elements $\mathcal{R}(n)_{\lambda}$ as a $GL_2(\mathbb{C})$-ideal. 
\end{itemize}
\end{theorem} 

\begin{remark}\label{remark:bonnafe} 
A minimal presentation of the algebra $\mathbb{C}[V\oplus V^*]^{D_{2n}}$  by generators and relations is given in \cite{alev-foissy} for $n=3,4,6$ and in \cite[Theorem 2.1]{bonnafe} for arbitrary $n$. Since  $V^*\cong V$ as $D_{2n}$-modules, these results can be translated to an explicit minimal presentation by generators and relations of 
$\mathbb{C}[V^2]^{D_{2n}}$. So the novel part of our Theorem~\ref{thm:m=2} 
is the nice explicit form of the relations $\mathcal{R}(n)_{\lambda}$  and the understanding of the $GL_2(\mathbb{C})$-module structure of the minimal syzygies. 
\end{remark} 

\begin{proofof}{Theorem~\ref{thm:m=2} (i)} We have $E_{1,2}.\mathcal{R}(n)_{n,2}=0$, $E_{1,1}.\mathcal{R}(n)_{n,2}=n\mathcal{R}(n)_{n,2}$, $E_{2,2}.\mathcal{R}(n)_{n,2}=2\mathcal{R}(n)_{n,2}$. Thus 
$\mathcal{R}(n)_{n,2}$ is a highest weight vector with weight $(n,2)$, and therefore  it generates an irreducible $GL_m(\mathbb{C})$-submodule of $\mathcal{F}(n,m)$ isomorphic to $S^{(n,2)}(\mathbb{C}^m)$ for $m\ge 2$. 
The elements $\rho_{2,0}$, $\rho_{1,1}^2-\rho_{2,0}\rho_{0,2}$ and 
$(-1)^k\frac 12\binom{2k}{k}\pi_{n-k,k}^2+\sum_{j=0}^{k-1}(-1)^j\binom{2k}{j}\pi_{n-j,j}\pi_{n-2k+j,2k-j}$ are annihilated by $E_{1,2}\in\mathfrak{gl}_2(\mathbb{C})$, so they are highest weight vectors with weights $(2,0)$, $(2,2)$ and 
$(2n-2k,2k)$. It follows that $\mathcal{R}(n)_{2n-2k,2k}$ is a highest weights vector of weight $(2n-2k,2k)$, and thus it generates an irreducible $GL_m(\mathbb{C})$-submodule of $\mathcal{F}(n,m)$ isomorphic to $S^{(2n-2k,2k)}(\mathbb{C}^m)$ for $m\ge 2$. 
The equality 
\begin{equation*}\label{eq:rel1gen} 
p_{n,0}q_{0,2}-2p_{n-1,1}q_{1,1}+p_{n-2,2}q_{2,0}=0\end{equation*} 
can be verified by direct computation, showing that 
$\varphi(n,2)(\mathcal{R}(n)_{n,2})=0$.  
To verify that $\mathcal{R}(n)_{2n-2k,2k}$ belongs to $\ker(\varphi(n,2))$ 
for $k=1,\dots,\lfloor \frac n2\rfloor$ we need to prove the equality  
\begin{equation}\label{eq:relkgen} 
(-1)^k\frac 12\binom{2k}{k}p_{n-k,k}^2+\sum_{j=0}^{k-1}(-1)^j\binom{2k}{j}p_{n-j,j}p_{n-2k+j,2k-j}-4^kq_{2,0}^{n-2k}(q_{1,1}^2-q_{2,0}q_{0,2})^k=0.  
\end{equation} 
Since $\mathcal{R}(n)_{2n-2k,2k}$ is a highest weight vector, it is fixed by the
subgroup $UT_2(\mathbb{C})$ of unipotent upper triangular matrices in $GL_2(\mathbb{C})$. Therefore $\varphi(n,2)(\mathcal{R}(n)_{2n-2k,2k})$ (the left hand side 
of \eqref{eq:relkgen}) is  a $UT_2(\mathbb{C})$-invariant  in 
$\mathbb{C}[V^2]$, and thus it  is constant along the $UT_2(\mathbb{C})$-orbits in $V^2$.   
The $UT_2(\mathbb{C})$-orbit of each point from a Zariski dense open subset in $V^2$ 
has non-empty intersection with the 
subset of lower triangular matrices  in 
$\mathbb{C}^{2\times 2}=V^2$. Therefore it is sufficient to show that 
$\sigma(\varphi(n,2)(\mathcal{R}(n)_{2n-2k,2k}))=0$, where $\sigma$ is the homomorphism 
$\mathbb{C}[x_1,y_1,x_2,y_2]\to \mathbb{C}[x_1,y_1,y_2]$ given by the specialization $x_2\mapsto 0$. Now $\sigma(p_{n-j,j})=y_1^{n-j}y_2^j$ for $j>0$, $\sigma(q_{0,2})=0$, 
$\sigma(q_{1,1})=\frac 12 x_1y_2$, hence 
\begin{equation}\label{eq:1relation}
\sigma(4^kq_{2,0}^{n-2k}(q_{1,1}^2-q_{2,0}q_{0,2})^k)=
x_1^ny_1^{n-2k}y_2^{2k}.\end{equation}
On the other hand, we have 
\begin{align}\label{eq:2relation}  
\sigma((-1)^k\frac 12\binom{2k}{k}p_{n-k,k}^2+\sum_{j=0}^{k-1}(-1)^j\binom{2k}{j}p_{n-j,j}p_{n-2k+j,2k-j})
\\ \notag = ((-1)^k\frac 12\binom{2k}{k}+\sum_{j=0}^{k-1}(-1)^j\binom{2k}{j})y_1^{2n-2k}y_2^{2k}+
x_1^ny_1^{n-2k}y_2^{2k}.
\end{align} 
Note that 
\begin{equation}\label{eq:3relation}
(-1)^k\frac 12\binom{2k}{k}+\sum_{j=0}^{k-1}(-1)^j\binom{2k}{j}=
\frac 12\sum_{j=0}^{2k}\binom{2k}{j}=\frac 12(1-1)^{2k}=0.\end{equation} 
The equalities \eqref{eq:1relation},\eqref{eq:2relation},\eqref{eq:3relation} 
show $\sigma(\varphi(n,2)(\mathcal{R}(n)_{2n-2k,2k}))=0$, implying in turn the equality \eqref{eq:relkgen},  and 
thus (i) is proved. 
\end{proofof} 

\begin{remark}\label{remark:elliott}
The expression $(-1)^k\frac 12\binom{2k}{k}\pi_{n-k,k}^2+\sum_{j=0}^{k-1}(-1)^j\binom{2k}{j}\pi_{n-j,j}\pi_{n-2k+j,2k-j}$ appears in the invariant theory of $m$-ary forms of degree $n$, see \cite[Section 114]{elliott}. 
Note also that this gives an explicit formula for the highest weight vector of the summand 
$S^{(2n-2k,2k)}(\mathbb{C}^m)$ of $S^2(S^n(\mathbb{C}^m))$in the decomposition given in Proposition~\ref{prop:well known GL-module isomorphisms} (iv). 
\end{remark} 

\subsection{Hironaka decomposition} 

The following statement is known, see for example 
\cite[Theorem 2.1]{bonnafe}. We shall present an alternative proof. 

\begin{proposition}\label{prop:Hilbert series m=2} 
The elements 
$p_{n,0}$, $q_{2,0}$, 
$p_{0,n}$, $q_{0,2}$  form a homogeneous system of parameters in $\mathbb{C}[V^2]^{D_{2n}}$, and $\mathbb{C}[V^2]^{D_{2n}}$ is a free module over the subalgebra 
$P(n,2):=\mathbb{C}[p_{n,0},q_{2,0},p_{0,n},q_{0,2}]$ generated by 
\begin{equation}\label{eq:secondary m=2} 
\{q_{1,1}^j, \ p_{n-i,i} \mid j=0,1,\dots,n; \quad i=1,\dots,n-1\}.\end{equation}
\end{proposition} 
\begin{proof} The common zero locus in $V^2$ of  
$p_{n,0}$, $q_{2,0}$, 
$p_{0,n}$, $q_{0,2}$ is the zero element of $V^2$, hence these polynomials  form a homogeneous system of parameters in $\mathbb{C}[V^2]^{D_{2n}}$ and in $\mathbb{C}[V^2]$ (so both of these algebras are finitely generated free modules over their subalgebra 
$P(n,2)$, see for example \cite[Section 2.3]{sturmfels}).  

Denote by $H$ the cyclic subgroup of $D_{2n}$ generated by 
$\left(\begin{array}{cc}\omega & 0 \\0 & \omega^{-1}\end{array}\right)$. 
Then $\mathbb{C}[V^2]^H$ is spanned as a $\mathbb{C}$-vector space by the monomials 
$x_1^{\alpha_1}x_2^{\alpha_2}y_1^{\beta_1}y_2^{\beta_2}$, where 
$\alpha_1+\alpha_2-\beta_1-\beta_2$ is divisible by $n$. One can easily deduce 
(see for example the method of the proof of Proposition~\ref{prop:cyclic}) 
that 
$\mathbb{C}[V^2]^H$ is a free $P(n,2)$-module generated by 
\begin{align*}\{(x_1y_2)^j,\ (x_2y_1)^j,\ x_1^{n-i}x_2^i,\ y_1^{n-i}y_2^i, \ x_1^n,\  x_2^n,
\ x_1^nx_2^n \mid 
j=0,1,\dots,n-1, \quad i=1,\dots,n-1\}.\end{align*}
Denote by $(P(n,2)_+)$ the ideal in $\mathbb{C}[V^2]$ generated by 
$p_{n,0}$, $q_{2,0}$, 
$p_{0,n}$, $q_{0,2}$. 
The other generator $\left(\begin{array}{cc} 0 & 1 \\ 1 & 0 \end{array}\right)$ of $D_{2n}$ 
maps the cosets of $x_1^n$ and $x_2^n$ in $\mathbb{C}[V^2]/(P(n,2)_+)$ to their negative, 
fixes the coset of $x_1^nx_2^n$ in $\mathbb{C}[V^2]/(P(n,2)_+)$,  permutes the other 
elements in the above $P(n,2)$-module generating system of $\mathbb{C}[V^2]^H$, 
and fixes all elements of $P(n,2)$. It follows that  
$\mathbb{C}[V^2]^{D_{2n}}$ is a free $P(n,2)$-module generated by 
\[\{(x_1y_2)^j+(y_1x_2)^j,\ x_1^{n-i}x_2^i+y_1^{n-i}y_2^i,\ x_1^nx_2^n \mid 
j=0,1,\dots,n-1, \quad i=1,\dots,n-1\}.\] 
Note finally that $(x_1y_2)^j+(y_1x_2)^j$ is congruent to $2^jq_{1,1}^j$ modulo 
$(P(n,2)_+)$, and  $x_1^nx_2^n$ is congruent to $-\frac 12(2q_{1,1})^n$ modulo 
$(P(n,2)_+)$. This shows that 
$\mathbb{C}[V^2]^{D_{2n}}$ is a free module over  $P(n,2)$ generated by \eqref{eq:secondary m=2}.  
\end{proof} 

\subsection{Further relations} 

Define $\mathcal{R}(n)_{n,2}^{n-j,2+j}$ for $j=0,1,\dots,n-2$ recursively by setting 
$\mathcal{R}(n)_{n,2}^{n,2}:=\mathcal{R}(n)_{n,2}$, and set 
\[\mathcal{R}(n)_{n,2}^{n-j-1,2+j+1}:=\frac{1}{n-2-j}E_{2,1}.\mathcal{R}(n)_{n,2}^{n-j,2+j}\] 
for $j=0,1,\dots,n-3$.  
Then $\{\mathcal{R}(n)_{n,2}^{n-j,2+j}\mid j=0,1,\dots,n-2\}$ is a $\mathbb{C}$-vector space basis in $\langle \mathcal{R}(n)_{n,2}\rangle_{GL_2(\mathbb{C})}$. Moreover, one shows by induction on $j$ that 
\begin{equation}\label{eq:R_{n,2}^(j)}
\mathcal{R}(n)_{n,2}^{n-j,2+j}=\pi_{n-j,j}\rho_{0,2}-2\pi_{n-j-1,j+1}\rho_{1,1}+\pi_{n-j-2,j+2}\rho_{2,0}.
\end{equation} 

\begin{proofof}{Theorem~\ref{thm:m=2} (ii)} 
The elements $q_{2,0},q_{1,1},q_{0,2}$ in $\mathbb{C}[V^2]^{D_{2n}}$ are algebraically independent. Therefore any non-zero element in $\ker(\varphi(n,2))$ involves a variable $\pi_{n-j,j}$ for some $j$. 
Moreover, a non-trivial linear combination of the elements $p_{n-j,j}\in \mathbb{C}[V^2]^{D_{2n}}$ $(j=0,1,\dots,n)$ involves a term $x_1^ix_2^{n-i}$ for some $i$, hence is not contained in the algebra generated by $q_{2,0},q_{1,1},q_{0,2}$ 
(all of whose terms have the same degree in the $x$-variables as in the $y$-variables). 
So 
every non-zero element in $\ker(\varphi(n,2))$ involves a term divisible by a monomial of the form $\pi_{n-j,j}\pi_{n-i,i}$ or $\pi_{n-j,j}\rho_{2-i,i}$. In particular, the minimal possible degree of an element of $\ker(\varphi(n,2))$ is $n+2$, the degree of $\mathcal{R}(n)_{n,2}$. 
It follows that a basis of $\langle \mathcal{R}(n)_{n,2}\rangle_{GL_2(\mathbb{C})}$ is contained in any minimal homogeneous generating system of $\ker(\varphi(n,2))$. 
Denote by $\mathcal{I}$ the ideal of $\mathcal{F}(n,2)$ generated by $\langle \mathcal{R}(n)_{n,2}\rangle_{GL_2(\mathbb{C})}$, and denote by $\mathcal{P}$ the subalgebra of $\mathcal{F}(n,2)$ generated by $\pi_{n,0},\rho_{2,0},\pi_{0,n},\rho_{0,2}$. The relations \eqref{eq:R_{n,2}^(j)} show that modulo $\mathcal{I}$, $\pi_{n-j,j}\rho_{1,1}$ is congruent to an element in the ideal of $\mathcal{F}(n,2)$ generated by $\pi_{n,0}$, $\pi_{0,n}$, 
$\rho_{2,0}$, $\rho_{0,2}$.
Consequently, 
any element of degree less than $2n$ in $\mathcal{F}(n,2)$ is congruent modulo $\mathcal{I}$ to an element of $\sum_{i=0}^{n-1}\mathcal{P}\rho_{1,1}^i+\sum_{j=1}^{n-1}\mathcal{P}\pi_{n-j,j}$. It follows by Proposition~\ref{prop:Hilbert series m=2} that a minimal homogeneous generating system of $\ker(\varphi)$ has no elements of degree $n+3,n+4,\dots,2n-1$. 

Set 
\[\mathcal{K}:=\bigoplus_{\lambda}\langle \mathcal{R}(n)_{\lambda}\rangle_{GL_2(\mathbb{C})}\] 
where $\lambda\in\{(2n-2,2),(2n-4,4),\dots,(2n-2\lfloor\frac n2\rfloor,
2\lfloor\frac n2\rfloor)\}$ (the above sum is direct, because by Theorem~\ref{thm:m=2} (i) the summands are pairwise non-isomorphic minimal $GL_2(\mathbb{C})$-invariant subspaces 
in $\mathcal{F}(n,2)$). 
Observe next that any term of any element of the ideal $\mathcal{I}$ of $\mathcal{F}(n,2)$  is divisible by a monomial of the form 
$\pi_{n-j,j}\rho_{2-i,i}$, whereas no term of the elements in $\mathcal{K}$ is divisible by a monomial of the form $\pi_{n-j,j}\rho_{2-i,i}$. It follows that $\mathcal{R}(n)_{n,2}^{n-j,2+j}$ $(j=0,\dots,n-2)$ together with a basis of $\mathcal{K}$ is part of some minimal homogeneous generating system of $\ker(\varphi(n,2))$. The dimension of $\mathcal{K}$ is 
$\sum_{i=1}^{\lfloor\frac n2\rfloor}(2n-4i+1)=\frac{n(n-1)}{2}$. 
On the other hand, it follows from Proposition~\ref{prop:Hilbert series m=2} that there must exist $\binom{n}{2}$ elements of degree $2n$ in $\ker(\varphi(n,2))$ that allow to rewrite 
the products $\pi_{n-j,j}\pi_{n-i,i}$ as an element of 
 $\sum_{i=0}^n\mathcal{P}\rho_{1,1}^i+\sum_{j=1}^{n-1}\mathcal{P}\pi_{n-j,j}$, 
 and these relations together with the earlier relations of degree $n+2$ are sufficient to generate $\ker(\varphi(n,2))$ up to degree $2n$. 
 Denote by $\mathcal{J}$ the ideal generated by $I\cup \mathcal{K}$. 
 Our observations so far imply that any element of $\mathcal{F}(n,2)$ is congruent modulo $\mathcal{J}$ to an element of 
$\sum_{i=0}^\infty \mathcal{P}\rho_{1,1}^i+\sum_{j=1}^{n-1}\mathcal{P}\pi_{n-j,j}$.  
It remains to show that $\rho_{1,1}^{n+1}$ is congruent modulo $\mathcal{J}$ to an element in 
$\sum_{i=0}^n\mathcal{P}\rho_{1,1}^i+\sum_{j=1}^{n-1}\mathcal{P}\pi_{n-j,j}$. When $n$ is even, consider the element  $\rho_{1,1}\mathcal{R}_{n,n}$. It belongs to $\mathcal{J}$ and has the term 
$-4^n\rho_{1,1}^{n+1}$.  The other terms of $\rho_{1,1}\mathcal{R}_{n,n}$ are congruent modulo $\mathcal{I}$ to an element of 
the ideal of $\mathcal{F}(n,2)$ generated by $\pi_{n,0}$, $\pi_{0,n}$, 
$\rho_{2,0}$, $\rho_{0,2}$ by \eqref{eq:R_{n,2}^(j)}. It follows that $\rho_{1,1}^{n+1}$ is congruent modulo $\mathcal{J}$ to an element of  
$\sum_{i=0}^n\mathcal{P}\rho_{1,1}^i+\sum_{j=1}^{n-1}\mathcal{P}\pi_{n-j,j}$.   
 So we are done when $n$ is even. For odd $n$ consider the element 
 $\rho_{1,1}(E_{2,1}.\mathcal{R}_{n+1,n-1})$ of $\mathcal{J}$. It has the term 
 $-2^n\rho_{1,1}^{n+1}$, and all the other terms are congruent modulo $\mathcal{I}$ to elements in 
 $\sum_{i=0}^n \mathcal{P}\rho_{1,1}^i+\sum_{j=1}^{n-1}\mathcal{P}\pi_{n-j,j}$.  
 This finishes the proof also for the case of odd $n$. 
  \end{proofof}

\section{The case $n=4$} \label{sec:n=4} 

\subsection{Secondary S-generating systems} \label{sec:secondary S-generating} 

In this subsection we return to the general setup of Section~\ref{sec:intro}. The action of the subgroup of diagonal matrices in $GL_m(\mathbb{C})$ induces an $\mathbb{N}_0^m$-grading on $\mathbb{C}[V^m]$ and  $\mathbb{C}[V^m]^G$. 
Write $\mathbb{C}[V^m]_{\alpha}$ and $\mathbb{C}[V^m]^G_{\alpha}$ for the multihomogeneous component of  $\mathbb{C}[V^m]$ and $\mathbb{C}[V^m]^G$ of multidegree $\alpha=(\alpha_1,\dots,\alpha_m)$. 

Now assume that $G$ is finite. Take a homogeneus system of parameters 
$p_1,\dots,p_k$ in $\mathbb{C}[V]^G$ (so $k=\dim(V)$). 
For $j=1,\dots,m$ and $i=1,\dots,k$ denote by $p_i^{(j)}$ the element of $\mathbb{C}[V^m]^G$ that maps $(v_1,\dots,v_m)\in V^m$ to $p_i(v_j)$. 
Then 
\begin{equation}\label{eq:general primary} 
p_1^{(j)},\dots,p_k^{(j)}\qquad (j=1,\dots,m) 
\end{equation} 
is a homogeneous system of parameters for $\mathbb{C}[V^m]^G$. 
This means that denoting by $P(m)$ the $\mathbb{C}$-subalgebra of $\mathbb{C}[V^m]^G$ 
generated by the (algebraically independent) elements \eqref{eq:general primary}, the space $\mathbb{C}[V^m]^G$ is a 
finitely generated free $P(m)$-module. The elements \eqref{eq:general primary} are 
called \emph{primary generators of} $\mathbb{C}[V^m]^G$, whereas a finite free $P(m)$-module generating system 
of $\mathbb{C}[V^m]^G$ is called a \emph{system of secondary generators} of 
$\mathbb{C}[V^m]^G$. 
By the Graded Nakayama Lemma a set of homogeneous elements in 
$\mathbb{C}[V^m]^G$ forms a system of secondary generators if and only if they form a basis in a vector space direct complement in $\mathbb{C}[V^m]^G$ of $(P(m)^+)$, the ideal in $\mathbb{C}[V^m]^G$ generated by the elements 
\eqref{eq:general primary}.

View the symmetric group $S_m$ as the subgroup of permutation matrices in $GL_m(\mathbb{C})$. Then $s\in S_m$ maps $\mathbb{C}[V^m]_{\alpha}$ to 
$\mathbb{C}[V^m]_{s\cdot \alpha}$, where $s\cdot \alpha=(\alpha_{s^{-1}(1)},\dots,\alpha_{s^{-1}(m)})$.

\begin{definition} \label{def:secondary S-generating system} 
 A subset $L$ of $\mathbb{C}[V^m]^G$ is called a \emph{system of secondary S-generators of} $\mathbb{C}[V^m]^G$ if 
 $L$ consists of multihomogeneous elements of decreasing multidegrees and
 \[L^*:=\{s\cdot f\mid f\in L,\quad s\in S_m/\mathrm{Stab}(\underline{\deg}(f))\},\]  
 is a system of  secondary generators  of $\mathbb{C}[V^m]^G$
 where $\underline{\deg}(f)\in\mathbb{N}_0^m$ stands for the multidegree of $f$ 
 and $S_m/\mathrm{Stab}(\underline{\deg}(f))$ stands for a chosen set of left coset representatives in $S_m$ with respect to the stabilizer subgroup $\mathrm{Stab}(\underline{\deg}(f))$ of the multidegree of $f$. 
 (Note that this notion depends on the choice $p_1,\dots,p_k$ of  homogeneous system of parameters in $\mathbb{C}[V]^G$.) 
\end{definition}

\begin{proposition}\label{prop:secondary S-generating system} 
A system of secondary S-generators of $\mathbb{C}[V^m]^G$ exists. 
\end{proposition} 

\begin{proof} 
 Since the elements in \eqref{eq:general primary} are multihomogeneous, the ideal 
 $(P(m)^+)$ of $\mathbb{C}[V^m]^G$ is spanned by multihomogenous elements, and therefore a system of secondary generators consisting of multihomogeneous elements exist. 
Moreover, a set $M$ of multihomogeneous elements forms a system of secondary generators if and only if for each multidegree $\alpha\in\mathbb{N}_0^m$ the subset $M\cap \mathbb{C}[V^m]_{\alpha}$ is a $\mathbb{C}$-vector space basis in a direct complement 
in $\mathbb{C}[V^m]^G_{\alpha}$ of  $(P(m)^+)\cap \mathbb{C}[V^m]_{\alpha}$. 

The action of $S_m$ preserves the set \eqref{eq:general primary}, hence it preserves the algebra $P(m)$ and the ideal  $(P(m)^+)$ in  $\mathbb{C}[V^m]^G$. An element $s\in S_m$ gives a vector space isomorphism between $\mathbb{C}[V^m]_{\alpha}$ and $\mathbb{C}[V^m]_{s\cdot\alpha}$, such that the subspace $\mathbb{C}[V^m]^G_{\alpha}$ is mapped onto $\mathbb{C}[V^m]^G_{s\cdot \alpha}$ and 
$\mathbb{C}[V^m]_{\alpha}\cap (P(m)^+)$ is mapped onto 
$\mathbb{C}[V^m]_{s\cdot \alpha}\cap (P(m)^+)$. 
Therefore to get 
a system of secondary S-generators of $\mathbb{C}[V^m]^G$ we just need to take for each decreasing multidegree $\alpha$ a basis in a direct complement in 
$\mathbb{C}[V^m]^G_{\alpha}$ of 
$\mathbb{C}[V^m]_{\alpha}\cap (P(m)^+)$, and the union of these sets as $\alpha$ ranges over all decreasing multidegrees will be a system of secondary S-generators. 
\end{proof}

\subsection{Hironaka decomposition for $n=4$, $m=3$} \label{sec:Hironaka for n=4, m=3} 

The elements 
\begin{equation}\label{eq:primary}  p_{4,0,0},p_{0,4,0},p_{0,0,4},q_{2,0,0},q_{0,2,0},q_{0,0,2}\end{equation}  
constitute a 
\emph{homogeneous system of parameters} in the algebras  $\mathbb{C}[V^3]$ and  $\mathbb{C}[V^3]^{D_8}$. We shall refer to the elements \eqref{eq:primary} as the \emph{primary generators of} $\mathbb{C}[V^3]^{D_8}$, and denote by 
$P(4,3)$ the subalgebra of $\mathbb{C}[V^3]$ generated by them. We are looking for an explicit free $P(4,3)$-module generating system (called \emph{system of secondary generators})  of $\mathbb{C}[V^3]^{D_8}$.

Denote by $H$ the cyclic subgroup of $D_8$ generated by 
$\left(\begin{array}{cc}\omega & 0 \\0 & \omega^{-1}\end{array}\right)$. 
The elements \eqref{eq:primary} constitute a homogeneous system of parameters in 
$\mathbb{C}[V^3]^H$ as well.  

\begin{proposition}\label{prop:cyclic} 
The following table gives a secondary S-generating system of $\mathbb{C}[V^3]^H$ 
with respect to $P(4,3)$:  
\[\begin{array}{c||c}
\text{multidegree} & \text{generator} \\
\hline\hline 
(0,0,0) & 1 \\
(1,1,0) & x_1y_2,\ y_1x_2 \\
(2,1,1) & x_1^2x_2x_3,\ y_1^2y_2y_3,\ x_1^2y_2y_3,\ y_1^2x_2x_3 \\
(2,2,0) & x_1^2x_2^2,\ y_1^2y_2^2,\ x_1^2y_2^2,\ y_1^2x_2^2 \\ 
(3,1,0) & x_1^3x_2,\ y_1^3y_2 \\ 
(4,0,0) & x_1^4 \\
(3,2,1) & x_1^3y_2^2y_3,\ y_1^3x_2^2x_3,\ x_1^3x_2^2y_3,\ y_1^3y_2^2x_3 \\ 
(3,3,0) & x_1^3y_2^3,\ y_1^3x_2^3 \\
(4,1,1) & x_1^4x_2y_3,\  x_1^4y_2x_3 \\ 
(3,3,2) & x_1^3x_2^3x_3^2,\ y_1^3y_2^3y_3^2, \ x_1^3x_2^3y_3^2,\ y_1^3y_2^3x_3^2 \\  
(4,2,2) & x_1^4x_2^2x_3^2,\ x_1^4y_2^2y_3^2, \ x_1^4x_2^2y_3^2,\ x_1^4y_2^2x_3^2 \\
(4,3,1) & x_1^4x_2^3x_3,\ x_1^4y_2^3y_3 \\ 
(4,4,0) & x_1^4x_2^4 \\ 
(4,3,3) & x_1^4x_2^3y_3^3,\ x_1^4y_2^3x_3^3 \\ 
(4,4,4) & x_1^4x_2^4x_3^4
\end{array}\]
\end{proposition}

\begin{proof} 
$\mathbb{C}[V^3]^H$ is spanned as a $\mathbb{C}$-vector space by the monomials 
$x_1^ix_2^jx_3^ky_1^ly_2^my_3^r$ such that $i+j+k-l-m-r$ is divisible by $4$. 
Consider the lexicographic monomial order in $\mathbb{C}[x,y]$ induced by the order $x<y$ of the variables. Then it is easy to check that the ideal generated by $xy$ and $x^4+y^4$ has the Gr\"obner basis $xy,x^4+y^4,x^5=x(x^4+y^4)-y^3(xy)$, hence the ideal generated by the initial monomials in $(xy,x^4+y^4)$ is $(xy,y^4,x^5)$. It follows that the monomials not divisible by any of 
$x_1y_1$, $x_2y_2$, $x_3y_3$, $y_1^4$, $y_2^4$, $y_3^4$, $x_1^5$, $x_2^5$, 
$x_3^5$ form a basis in a vector space direct complement of the ideal in $\mathbb{C}[V^3]$ generated by the elements in \eqref{eq:primary}. 
As each monomial spans an $H$-invariant subspace in $\mathbb{C}[V^3]$, 
the $H$-invariant monomials among them form a system of secondary generators for 
$\mathbb{C}[V^3]^H$, and the table above contains all those with decreasing multidegree. 
\end{proof} 

\begin{proposition}\label{prop:Hironaka n=4, m=3} 
The following table gives a secondary S-generating system of $\mathbb{C}[V^3]^{D_8}$: 
\[\begin{array}{c||c}
\text{multidegree} & \text{generator} \\
\hline\hline 
(0,0,0) & 1 \\
(1,1,0) & q_{1,1,0} \\
(2,1,1) & p_{2,1,1},\ q_{1,1,0}q_{1,0,1} \\
(2,2,0) & p_{2,2,0},\ q_{1,1,0}^2 \\ 
(3,1,0) & p_{3,1,0} \\ 
(3,2,1) & p_{3,1,0}q_{0,1,1},\ q_{1,1,0}^2q_{1,0,1} \\ 
(3,3,0) & q_{1,1,0}^3 \\
(4,1,1) & p_{3,1,0}q_{1,0,1} \\ 
(3,3,2) & p_{2,1,1}p_{1,2,1},  \ p_{3,1,0}q_{0,1,1}^2 \\  
(4,2,2) &  p_{3,1,0}q_{1,0,1}q_{0,1,1},\ q_{1,1,0}^2q_{1,0,1}^2 \\
(4,3,1) & q_{1,1,0}^3q_{1,0,1}  \\ 
(4,4,0) & q_{1,1,0}^4 \\ 
(4,3,3) & p_{3,1,0}q_{1,0,1}q_{0,1,1}^2 
\end{array}\]
\end{proposition} 

\begin{proof} 
Following the notation of Section~\ref{sec:secondary S-generating}, $(P(4,3)^+)$ stands for the ideal in $\mathbb{C}[V^3]^{D_8}$ generated by the elements \eqref{eq:primary}. 
The element $\left(\begin{array}{cc} 0 & 1 \\ 1 & 0 \end{array}\right)$ of $D_8$ permutes up to sign the cosets modulo $(P(4,3)^+)$ of the monomials in the table in Proposition~\ref{prop:cyclic} (note that $x^4$ is congruent to $-y^4$  modulo $(P(4,3)^+)$). It follows by Proposition~\ref{prop:cyclic} that a $D_8$-invariant direct complement of $(P(4,3)^+)$ in the sum of the homogeneous components of $\mathbb{C}[V^3]^H$ with decreasing multidegree has the basis 
$C_+\cup C_-$, where 
\begin{align*}C_+=\{1, x_1y_2+y_1x_2, 
x_1^2x_2x_3+y_1^2y_2y_3, 
x_1^2y_2y_3+y_1^2x_2x_3, 
x_1^2x_2^2+y_1^2y_2^2, 
x_1^2y_2^2+y_1^2x_2^2, 
\\ x_1^3x_2+y_1^3y_2, 
x_1^3y_2^2y_3+y_1^3x_2^2x_3, 
x_1^3x_2^2y_3+y_1^3y_2^2x_3, 
 x_1^3y_2^3+y_1^3x_2^3, 
 x_1^4x_2y_3-x_1^4y_2x_3, 
 \\ 
 x_1^3x_2^3x_3^2+y_1^3y_2^3y_3^2, 
 x_1^3x_2^3y_3^2+y_1^3y_2^3x_3^2, 
 x_1^4x_2^2x_3^2-x_1^4y_2^2y_3^2, 
 x_1^4x_2^2y_3^2-x_1^4y_2^2x_3^2 , 
 \\ 
 x_1^4x_2^3x_3-x_1^4y_2^3y_3, 
x_1^4x_2^4, 
x_1^4x_2^3y_3^3-x_1^4y_2^3x_3^3 \}
\end{align*}
and 
\begin{align*} 
C_-=\{x_1y_2-y_1x_2, 
x_1^2x_2x_3-y_1^2y_2y_3, 
x_1^2y_2y_3-y_1^2x_2x_3, 
x_1^2x_2^2-y_1^2y_2^2, 
x_1^2y_2^2-y_1^2x_2^2, 
\\ x_1^3x_2-y_1^3y_2, 
x_1^3y_2^2y_3-y_1^3x_2^2x_3, 
x_1^3x_2^2y_3-y_1^3y_2^2x_3, 
 x_1^3y_2^3-y_1^3x_2^3, 
 x_1^4x_2y_3+x_1^4y_2x_3, 
 \\ 
 x_1^3x_2^3x_3^2-y_1^3y_2^3y_3^2, 
 x_1^3x_2^3y_3^2-y_1^3y_2^3x_3^2, 
 x_1^4x_2^2x_3^2+x_1^4y_2^2y_3^2, 
 x_1^4x_2^2y_3^2+x_1^4y_2^2x_3^2 , 
 \\ 
 x_1^4x_2^3x_3+x_1^4y_2^3y_3, 
x_1^4x_2^3y_3^3+x_1^4y_2^3x_3^3 , 
x_1^4, 
x_1^4x_2^4x_3^4\}. 
\end{align*}
The elements in $C_+$ are $D_8$-invariant, whereas the elements in $C_-$ span a $1$-dimensional $D_8$-invariant subspace on which $D_8$ acts via the determinant representation. It follows that $C_+$ is a system of secondary S-generators for 
$\mathbb{C}[V^3]^{D_8}$. It is easy to see that modulo the ideal 
$(P(4,3)^+)$ the elements listed in the table in the statement of our proposition agree with non-zero scalar multiples of the elements in $C_+$. 
\end{proof} 

\subsection{Relations for $n=4$ and $m=3$} 

\begin{proposition}\label{prop:n=4, m=3} 
The $GL_3(\mathbb{C})$-ideal $\ker(\varphi(4,3))$ is minimally generated by 
$\mathcal{R}_{2,2,2}$, $\mathcal{R}(4)_{4,2}$, $\mathcal{R}(4)_{6,2}$, and $\mathcal{R}(4)_{4,4}$. 
\end{proposition} 

\begin{remark}
We have $\dim(S^{(4,2)}(\mathbb{C}^3))=27$, 
$\dim(S^{(6,2)}(\mathbb{C}^3))=60$, 
$\dim(S^{(4,4)}(\mathbb{C}^3))=15$, 
and $\dim(S^{(2,2,2)}(\mathbb{C}^3))=1$. 
Therefore Proposition~\ref{prop:n=4, m=3} implies that a minimal homogeneous generating system of the ideal $\ker(\varphi(4,3))$ consists of $103$ elements. 
For comparison we mention that a minimal homogeneous generating system of 
$\ker(\varphi(4,2))$ consists of $3+5+1=9$ elements. 
\end{remark}

The logic of the proof of Proposition~\ref{prop:n=4, m=3} is furnished by the following general lemma: 

\begin{lemma}\label{lemma:furnish} 
Let $\varphi:\mathcal{F}\to R$ be a surjective homomorphism of graded $\mathbb{C}$-algebras, where $R$ is a connected graded Cohen-Macaulay algebra (cf. \cite[Section 2.3]{sturmfels}), so there exist 
homogeneous elements $h_1,\dots,h_k$ and $t_1,\dots,t_l$ in $\mathcal{F}$ with the following properties: 
\begin{enumerate} 
\item $\varphi(h_1),\dots,\varphi(h_k)$ is a homogeneous system of parameters in $R$. 
\item $R=P\varphi(t_1)\oplus\cdots\oplus P\varphi(t_l)$ where $P=\mathbb{C}[\varphi(h_1),\dots,\varphi(h_l)]$. 
\end{enumerate} 
Let $\mathcal{K}$ be a homogeneous ideal in $\mathcal{F}$ contained in $\ker(\varphi)$, and assume that  for some $d\in \mathbb{N}$ we have 
\[\mathcal{F}_{\le d}=\mathrm{Span}_{\mathbb{C}}\mathcal{T}_{\le d}+\mathcal{H}_{\le d}+\mathcal{K}_{\le d}\] 
where $\mathcal{H}$ is the ideal in $\mathcal{F}$ generated by $h_1,\dots,h_k$, 
$\mathcal{T}:=\{t_1,\dots,t_l\}$, 
$d\in \mathbb{N}$ and for a subset $A$ of homogeneous elements (respectively graded subspace) in $\mathcal{F}$ we write $A_{\le d}$ for the set of elements (respectively sum of homogeneous components) of $A$ with degree $\le d$. 

Then we have $\mathcal{K}_{\le d}=\ker(\varphi)_{\le d}$. 
\end{lemma} 

\begin{proof} 
One can show by an induction on the degree that our assumptions imply the equality 
\[\mathcal{F}_{\le d}=\mathcal{K}_{\le d}+\mathcal{P}_{\le d-\deg(t_1)}t_1+\cdots+ \mathcal{P}_{\le d-\deg(t_l)}t_l,\] 
where $\mathcal{P}=\mathbb{C}[h_1,\dots,h_k]$. Moreover, the restriction of $\varphi$ to 
$\mathcal{P}$ and to $\mathcal{P}t_1+\cdots+ \mathcal{P}t_l$ is injective. This clearly implies the desired equality $\mathcal{K}_{\le d}=\ker(\varphi)_{\le d}$. 
\end{proof}

\begin{proofof}{Proposition~\ref{prop:n=4, m=3}} 
For a decreasing multidegree $\alpha=(\alpha_1,\alpha_2,\alpha_3)$ denote by $\mathcal{T}_{\alpha}$ the set of products of the variables $\pi_{i_1,i_2,i_3}$, $\rho_{j_1,j_2,j_3}$ having  multidegree $\alpha$ that correspond to the elements of multidegree $\alpha$ in the secondary S-generating system of $\mathbb{C}[V^3]^{D_8}$ given in Proposition~\ref{prop:secondary S-generating system}. For example, $\mathcal{T}_{1,1,1}=\emptyset$ and 
$\mathcal{T}_{4,2,2}=\{\pi_{3,1,0}\rho_{1,0,1}\rho_{0,1,1},\rho_{1,1,0}^2\rho_{1,0,1}^2\}$. 
Denote by $\mathcal{K}$ the $GL$-ideal in $\mathcal{F}_3$ generated by  
$\mathcal{R}_{2,2,2}$, $\mathcal{R}(4)_{4,2}$, $\mathcal{R}(4)_{6,2}$, and $\mathcal{R}(4)_{4,4}$, and write $\mathcal{H}$ for the ideal in $\mathcal{F}(4,3)$ generated by 
$\pi_{4,0,0},\pi_{0,4,0},\pi_{0,0,4},\rho_{2,0,0},\rho_{0,2,0},\rho_{0,0,2}$. 
The subspaces $\mathcal{K}$ and $\mathcal{H}$ are spanned by multihomogeneous elements of $\mathcal{F}(4,3)$, and we shall write $\mathcal{K}_{\alpha}$, $\mathcal{H}_{\alpha}$, $\mathcal{F}(4,3)_{\alpha}$ for their components of multidegree $\alpha$. 
We shall show that 
\begin{equation}\label{eq:T_alpha+K_alpha+H_alpha} 
\mathrm{Span}_{\mathbb{C}}\{\mathcal{T}_{\alpha}\}+\mathcal{K}_{\alpha}+\mathcal{H}_{\alpha}=\mathcal{F}(4,3)_{\alpha} 
\text{ for all decreasing }\alpha\in\mathbb{N}_0^3 \text{ with }\alpha_1+\alpha_2+\alpha_3\le 10.
\end{equation} 
Note that $\mathcal{H}$ is not a $GL$-ideal in $\mathcal{F}(4,3)$, however, it is preserved by the subgroup $S_3$ of $GL_3(\mathbb{C})$. For $s\in S_3$ we have 
$s\cdot \mathcal{H}_{\alpha}=\mathcal{H}_{s\cdot\alpha}$ and 
$s\cdot \mathcal{K}_{\alpha}=\mathcal{K}_{s\cdot\alpha}$. 
Thus \eqref{eq:T_alpha+K_alpha+H_alpha} implies 
$\mathcal{F}_{s\cdot \alpha}=\mathrm{Span}_{\mathbb{C}}\{s\cdot \mathcal{T}_{\alpha}\}
+\mathcal{H}_{s\cdot\alpha}+\mathcal{K}_{s\cdot\alpha}$. 
Set 
\[\mathcal{T}:=\bigsqcup_{\alpha\in \mathbb{N}_0^3,\ \alpha_1\ge\alpha_2\ge\alpha_3}
\bigsqcup_{s\in S_3/\mathrm{Stab}_{S_3}(\alpha)}s\cdot\mathcal{T}_{\alpha}.\] 
Then the generators of $\mathcal{H}$ are mapped to a homogeneous system of parameters in $R:=\mathbb{C}[V^3]^{D_8}$ and $\mathcal{T}$ is mapped by $\varphi(4,3)$ to a system of secondary generators of $R$ (see Definition~\ref{def:secondary S-generating system}).  
 Given \eqref{eq:T_alpha+K_alpha+H_alpha}  the above considerations show that the assumptions of Lemma~\ref{lemma:furnish} hold for $\mathcal{F}=\mathcal{F}(4,3)$,  
 $\varphi=\varphi(4,3)$ and $d=10$, and therefore by Lemma~\ref{lemma:furnish} we conclude 
 that $\mathcal{K}_{\le 10}=\ker(\varphi(4,3))_{\le 10}$. 
 Now $\ker(\varphi(4,3))$ is generated by its elements of degree at most $10$ by 
 Proposition~\ref{prop:2n+2}, implying the equality $\mathcal{K}=\ker(\varphi(4,3))$. 
 
 It remains to prove \eqref{eq:T_alpha+K_alpha+H_alpha}. For multidegrees $\alpha$ with $\alpha_3=0$ this was done in the proof of Theorem~\ref{thm:m=2}. Note also that  
 $\mathcal{F}_{\alpha}=\{0\}$ if $\alpha_1+\alpha_2+\alpha_3$ is odd. 
 
 Denote by $\mathcal{M}_{\alpha}$ the products of multidegree $\alpha$ in 
 $\{\pi_{\beta},\rho_{\gamma}\mid \beta_i<4,\gamma_j<2\}$ (the generators of $\mathcal{F}(4,3)$ different from $\pi_{4,0,0},\pi_{0,4,0},\pi_{0,0,4},\rho_{2,0,0},\rho_{0,2,0},\rho_{0,0,2}$).  
 In the first table below we collect for all decreasing $\alpha\in \mathbb{N}_0^3$ with 
 $\alpha_3>0$ and $\sum_{i=1}^3\alpha_i\le 6$ even the elements of 
 $\mathcal{M}_{\alpha}$.  
 \[\begin{array}{c||c|c}
\alpha & \mathcal{T}_{\alpha} & \mathcal{M}_{\alpha}\setminus \mathcal{T}_{\alpha} \\
\hline \hline 
(2,1,1) & \pi_{2,1,1},\ \rho_{1,1,0}\rho_{1,0,1} & 
\\ \hline
(4,1,1) & \pi_{3,1,0}\rho_{1,0,1} &  \pi_{3,0,1}\rho_{1,1,0} 
\\ \hline 
(3,2,1) & \rho_{1,1,0}^2\rho_{1,0,1},\ \pi_{3,1,0}\rho_{0,1,1} 
& \pi_{2,1,1}\rho_{1,1,0},\ \pi_{2,2,0}\rho_{1,0,1} 
\\ \hline 
(2,2,2) & & \rho_{1,1,0}\rho_{1,0,1}\rho_{0,1,1},\ \pi_{2,1,1}\rho_{0,1,1},\ 
\pi_{1,2,1}\rho_{1,0,1},\ \pi_{1,1,2}\rho_{1,1,0}
\\ \hline 
\end{array} \]

 Now going multidegree by multidegree we shall show that the products 
 $\mathcal{M}_{\alpha}\setminus\mathcal{T}_{\alpha}$ are congruent modulo 
 $\mathcal{H}+\mathcal{K}$ to a $\mathbb{C}$-linear combination of the elements in $\mathcal{T}_{\alpha}$. 
 Denote by $\mathcal{I}$ the 
$GL$-ideal  of $\mathcal{F}(4,3)$ generated by $\mathcal{R}_{2,2,2}$ and 
$\mathcal{R}(4)_{4,2}$, and we shall write $a\equiv_{\mathcal{I}}b$ 
for some $a,b\in \mathcal{F}(4,3)$ if $a-b\in \mathcal{H}+\mathcal{I}$. 
 We shall write 
 $a\equiv b$ for $a,b\in\mathcal{F}(4,3)$ if $a-b\in\mathcal{H}+\mathcal{K}$. 
 Note that $\mathcal{I}\subseteq \mathcal{K}$ and therefore 
 $a\equiv_{\mathcal{I}}b$ implies $a\equiv b$. 
 
 There is nothing to do for the multidegree $(2,1,1)$. Taking into account Theorem~\ref{thm:m=2} this shows also that for $d<6$ we have $\ker(\varphi(4,3))_d=\{0\}$.  
 
$\mathbf{(4,1,1)}$: The relation \eqref{eq:R(n,2)} in the special case $n=4$, $m=3$ is 
\begin{equation}\label{eq:R_{4,2,0}}
\mathcal{R}(4)_{4,2}:=\pi_{4,0,0}\rho_{0,2,0}-2\pi_{3,1,0}\rho_{1,1,0}+\pi_{2,2,0}\rho_{2,0,0}. 
\end{equation}
Applying $E_{3,2}\in \mathfrak{gl}_3(\mathbb{C})$ to $\mathcal{R}(4)_{4,2}$ we get the following element of $\langle \mathcal{R}(4)_{4,2}\rangle_{GL_3(\mathbb{C})}$: 
\begin{equation}\label{eq:R__{4,2}^{4,1,1}}
\mathcal{R}(4)_{4,2}^{4,1,1}:=\frac 12E_{3,2}.\mathcal{R}_{4,2,0}=
\pi_{4,0,0}\rho_{0,1,1}-\pi_{3,1,0}\rho_{1,0,1}-\pi_{3,0,1}\rho_{1,1,0}+\pi_{2,1,1}\rho_{2,0,0}
\end{equation}
This relation implies that 
\begin{equation}\label{eq:p_301q_110} 
\pi_{3,0,1}\rho_{1,1,0}\equiv_{\mathcal{I}} -\pi_{3,1,0}\rho_{1,0,1}.  
\end{equation} 

$\mathbf{(3,2,1)}$: Consider the following elements of $\langle \mathcal{R}(4)_{4,2}\rangle_{GL_3(\mathbb{C})}$: 

\begin{align}\label{eq:R_{4,2}^{3,2,1}(1)}
\mathcal{R}(4)_{4,2}^{3,2,1}(1):=\frac 12E_{3,1}.\mathcal{R}(4)_{4,2}=
-\pi_{3,1,0}\rho_{0,1,1}-3\pi_{2,1,1}\rho_{1,1,0}
+\pi_{2,2,0}\rho_{1,0,1} 
\\ \notag +2\pi_{3,0,1}\rho_{0,2,0}+\pi_{1,2,1}\rho_{2,0,0}
\end{align}

\begin{align}\label{eq:R_{4,2}^{3,2,1}(2)}
\mathcal{R}(4)_{4,2}^{3,2,1}(2):=E_{2,1}.\mathcal{R}(4)_{4,2}^{4,1,1}=
& 3\pi_{3,1,0}\rho_{0,1,1}-\pi_{2,1,1}\rho_{1,1,0} 
-3\pi_{2,2,0}\rho_{1,0,1}
\\ \notag &+2\pi_{1,2,1}\rho_{2,0,0}-\pi_{3,0,1}\rho_{0,2,0}
\end{align} 

Now $3\mathcal{R}(4)_{4,2}^{3,2,1}(1)+\mathcal{R}(4)_{4,2}^{3,2,1}(2)\in\mathcal{K}$ implies that 
\begin{equation}\label{eq:p_211q_110} 
\pi_{2,1,1}\rho_{1,1,0}\equiv_{\mathcal{I}} 0
\end{equation} 
and \eqref{eq:R_{4,2}^{3,2,1}(1)} and \eqref{eq:p_211q_110} imply 
\begin{equation}\label{eq:p_220q_101} 
\pi_{2,2,0}\rho_{1,0,1}\equiv_{\mathcal{I}} \pi_{3,1,0}\rho_{0,1,1}. 
\end{equation} 

$\mathbf{(2,2,2)}$: $\mathcal{R}_{2,2,2}\in\mathcal{I}$ implies 
\begin{equation}\label{eq:q_110q_101q_011} 
\rho_{1,1,0}\rho_{1,0,1}\rho_{0,1,1}\equiv_{\mathcal{I}} 0. 
\end{equation} 
Set $U:=\langle \mathcal{R}(4)_{4,2}\rangle_{GL_3(\mathbb{C})}$. Then 
$\dim U_{2,2,2}=3$ (since there are $3$ semistandard tableaux of shape $(4,2)$ and content $1,1,2,2,3,3$).   
The table above shows that $|\mathcal{M}_{2,2,2}\setminus\mathcal{T}_{2,2,2}|=4$, hence by Lemma~\ref{lemma:furnish} we conclude that 
$\dim \ker(\varphi(4,3))_{2,2,2}=4$. It follows that 
\[\ker(\varphi(4,3))_{2,2,2}=\mathbb{C} \mathcal{R}_{2,2,2}\oplus U_{2,2,2}.\] 
Consequently, $\ker(\varphi(4,3))_{2,2,2}=\mathcal{I}_{2,2,2}$,  and  
\begin{equation}\label{eq:p_211q_011}
\pi_{2,1,1}\rho_{0,1,1}\equiv_{\mathcal{I}} 0, 
\quad \pi_{1,2,1}\rho_{1,0,1}\equiv_{\mathcal{I}} 0, 
\quad  \pi_{1,1,2}\rho_{1,1,0}\equiv_{\mathcal{I}} 0. 
\end{equation} 

\bigskip
The calculations so far show that 
\[\ker(\varphi(4,3))_6=\mathbb{C} \mathcal{R}_{2,2,2}\oplus \langle \mathcal{R}(4)_{4,2}\rangle_{GL_3(\mathbb{C})}.\] 
Moreover, for later reference we mention some consequences of the relations of degree $6$.  
The relations \eqref{eq:R_{n,2}^(j)} in the special case $n=4$ give 
that 
\[\pi_{3,1,0}\rho_{1,1,0}\equiv_{\mathcal{I}} 0, \quad   \pi_{2,2,0}\rho_{1,1,0}\equiv_{\mathcal{I}} 0.\] 
Denote by $A$ the union of the $S_3$-orbits of 
$\pi_{2,1,1}\rho_{1,1,0}$, $\pi_{2,1,1}\rho_{0,1,1}$, $\rho_{1,1,0}\rho_{1,0,1}\rho_{0,1,1}$, 
 $\pi_{3,1,0}\rho_{1,1,0}$ and  $\pi_{2,2,0}\rho_{1,1,0}$. 
Each element of $A$ 
is congruent to zero modulo $\mathcal{H}+\mathcal{I}$ by 
\eqref{eq:p_211q_110}, \eqref{eq:q_110q_101q_011}, \eqref{eq:p_211q_011} 
(and taking into account that the ideal $\mathcal{H}+\mathcal{I}$ is preserved by the action 
of $S_3$ on $\mathcal{F}(4,3)$). 
Denote by $B$ the $S_3$-orbit of  $\pi_{2,2,0}\rho_{1,0,1}$. 
By \eqref{eq:p_220q_101} and as $\mathcal{I}+\mathcal{H}$ is $S_3$-stable, we conclude that each element 
of $B$ is congruent modulo $\mathcal{H}+\mathcal{I}$ to an element of the $S_3$-orbit of 
$\pi_{3,1,0}\rho_{0,1,1}$.  
Set $C:=\{\pi_{3,0,1}\rho_{1,1,0}\}$,  
by \eqref{eq:p_301q_110} the  element 
of $C$ is congruent modulo $\mathcal{H}+\mathcal{I}$ to  
$-\pi_{3,1,0}\rho_{0,1,1}$. 
Summarizing, we have that the factor space 
\begin{align}\label{eq:ABC} 
\mathcal{F}(4,3)/(\mathcal{I}+\mathcal{H}) \text{ is spanned over }\mathbb{C} 
\text{ by products of the variables }\pi_{i_1,i_2,i_3},\ \rho_{j_1,j_2,j_3} 
\\ \notag  \text{ not divisible by any element of } A\cup B\cup C.
\end{align}

Now we turn to the relations of degree $8$. Every non-zero term in any element of the 
$GL$-ideal $\mathcal{I}$ of $\mathcal{F}(4,3)$ generated by $\mathcal{R}_{2,2,2}$ and 
$\mathcal{R}(4)_{4,2}$ involves a variable $\rho_{i,j,k}$, and every non-zero term of any element of $\mathcal{H}$ involves a variable from $\{\pi_{4,0,0},\pi_{0,4,0},\pi_{0,0,4},
\rho_{2,0,0},\rho_{0,2,0},\rho_{0,0,2}\}$. Hence 
none of $\mathcal{R}(4)_{6,2}$ or $\mathcal{R}(4)_{4,4}$ is contained in $\mathcal{I}+\mathcal{H}$. 
Consequently, by basic principles about semisimple representations we have 
\begin{equation}\label{eq:disjoint}
(\mathcal{I}+\mathcal{H})\cap \langle \mathcal{R}(4)_{6,2},\  \mathcal{R}(4)_{4,4}\rangle_{GL_3(\mathbb{C})}=\{0\}. 
\end{equation}   
For the multidegrees $\alpha$ with $\sum \alpha_i=8$ we shall first show that each element 
of $\mathcal{M}_{\alpha}\setminus \mathcal{T}_{\alpha}$ involving a variable $\rho_{i,j,k}$ is congruent to an element of $\mathrm{Span}_{\mathbb{C}} \mathcal{T}_{\alpha}$ modulo 
$\mathcal{I}+\mathcal{H}$. 
In fact by \eqref{eq:ABC} it is sufficient to do it for the elements of the subset 
$\mathcal{U}_{\alpha}$ of $\mathcal{M}_{\alpha}\setminus \mathcal{T}_{\alpha}$ not divisible by any element of $A\cup B\cup C$. 
Denote by $\mathcal{N}_{\alpha}$ the set of  elements of 
$\mathcal{M}_{\alpha}\setminus \mathcal{T}_{\alpha}$  of the form 
$\pi_{i_1,i_2,i_3}\pi_{j_1,j_2,j_3}$. In the table below 
we collect for all decreasing $\alpha\in \mathbb{N}_0^3$ with 
 $\alpha_3>0$ and $\sum_{i=1}^3\alpha_i=8$ the elements of 
 $\mathcal{T}_{\alpha}$, $\mathcal{U}_{\alpha}$, $\mathcal{N}_{\alpha}$:  

\[\begin{array}{c||c|c|c}
\alpha & \mathcal{T}_{\alpha} & \mathcal{U}_{\alpha}\setminus \mathcal{N}_{\alpha}  &\mathcal{N}_{\alpha} 
\\ \hline \hline 
(6,1,1) & & &\pi_{3,1,0}\pi_{3,0,1} 
\\ \hline 
(5,2,1) &  & 
& \pi_{3,1,0}\pi_{2,1,1},\ \pi_{3,0,1}\pi_{2,2,0} 
\\ \hline 
(4,3,1) & \rho_{1,1,0}^3\rho_{1,0,1} & 
& 
\pi_{3,1,0}\pi_{1,2,1},\ \pi_{2,2,0}\pi_{2,1,1},\ \pi_{3,0,1}\pi_{1,3,0} 
\\ \hline 
(4,2,2) & \rho_{1,1,0}^2\rho_{1,0,1}^2,\ \pi_{3,1,0}\rho_{1,0,1}\rho_{0,1,1} & 
&  \pi_{3,1,0}\pi_{1,1,2},\ \pi_{3,0,1}\pi_{1,2,1}, 
\\ & & &  \pi_{2,2,0}\pi_{2,0,2},\ \pi_{2,1,1}^2
\\ \hline 
(3,3,2) &  \pi_{3,1,0}\rho_{0,1,1}^2,\ \pi_{2,1,1}\pi_{1,2,1} &
\pi_{1,3,0}\rho_{1,0,1}^2 
 & \pi_{3,1,0}\pi_{0,2,2},\ \pi_{3,0,1}\pi_{0,3,1},
 \\ & & &  \pi_{1,3,0}\pi_{2,0,2},\ 
\pi_{2,2,0}\pi_{1,1,2}  
\\ \hline 
\end{array} \]
We see that $\mathcal{U}_{\alpha}\neq\emptyset$ only for $\alpha=(3,3,2)$, 
and then 
$\mathcal{U}_{3,3,2}=\{\pi_{1,3,0}\rho_{1,0,1}^2\}$. 
Consider 
\begin{align}\label{eq:R_{4,2}^{3,3,0}}
\mathcal{R}(4)_{4,2}^{3,3,0}:=E_{2,1}.\mathcal{R}(4)_{4,2}^{4,2,0} 
=&-4\pi_{2,2,0}\rho_{1,1,0}
\\ \notag &+2\pi_{3,1,0}\rho_{0,2,0}+
2\pi_{1,3,0}\rho_{2,0,0}
\end{align}
and 
\begin{align}\label{eq:R_{4,2}^{2,3,1}}
\mathcal{R}(4)_{4,2}^{2,3,1}:=E_{3,1}.\mathcal{R}(4)_{4,2}^{3,3,0}
=-8\pi_{1,2,1}\rho_{1,1,0}-4\pi_{2,2,0}\rho_{0,1,1}+4\pi_{1,3,0}\rho_{1,0,1}
\\ \notag +6\pi_{2,1,1}\rho_{0,2,0}+2\pi_{0,3,1}\rho_{2,0,0} 
\\ \label{eq:bR_{4,2}^{2,3,1}} \equiv_{\mathcal{I}}-4\pi_{2,2,0}\rho_{0,1,1}+4\pi_{1,3,0}\rho_{1,0,1} 
\end{align} 
(the congruence \eqref{eq:bR_{4,2}^{2,3,1}} follows from \eqref{eq:p_211q_110} and the fact that $\mathcal{I}+\mathcal{H}$ is $S_3$-stable). 

The element $\rho_{1,0,1}\mathcal{R}(4)_{4,2}^{2,3,1}\in \mathcal{I}$ 
shows that 
\[\pi_{1,3,0}\rho_{1,0,1}^2\equiv_{\mathcal{I}} \pi_{2,2,0}\rho_{0,1,1}\rho_{1,0,1}
\equiv_{\mathcal{I}} \pi_{3,1,0}\rho_{0,1,1}^2\] 
(we used also \eqref{eq:p_220q_101}).  

At this stage by  \eqref{eq:disjoint} to prove that each element of $\mathcal{N}_{\alpha}$  is congruent to an element of  
 $\mathrm{Span}_{\mathbb{C}} \mathcal{T}_{\alpha}$ modulo $\mathcal{K}+\mathcal{H}$ it is sufficient to verify that 
 $|\mathcal{N}_{\alpha}|$ equals the dimension of the component of multidegree 
 $\alpha$ of $\langle \mathcal{R}(4)_{6,2}\rangle_{GL_3(\mathbb{C})}\oplus 
\langle \mathcal{R}(4)_{4,4}\rangle_{GL_3(\mathbb{C})}$. 
Recall that  $\langle \mathcal{R}(4)_{6,2}\rangle_{GL_3(\mathbb{C})}\cong S^{(4,2)}(\mathbb{C}^3)$ 
and $\langle \mathcal{R}(4)_{4,4}\rangle_{GL_3(\mathbb{C})}\cong S^{(4,4)}(\mathbb{C}^3)$ as $GL_3(\mathbb{C})$-modules, and under these isomorphisms the multihomogeneous component of multidegree $\alpha$ corresponds to the weight subspace with weight $\alpha$.   
The dimension of $S^{\lambda}(\mathbb{C}^3)_{\alpha}$ (the weight subspace with weight $\alpha$) equals the number of semi-standard tableaux of shape $\lambda$ and content $\alpha$. The dimensions relevant to us are given in the following table:  
\[\begin{array}{c||c|c|c|c|c} 
\alpha & (6,1,1) & (5,2,1)& (4,3,1) & (4,2,2) & (3,3,2) 
 \\ \hline \hline 
\dim(S^{(6,2)}(\mathbb{C}^3)_{\alpha}) & 1 & 2 & 2 & 3 & 3 
\\ \hline
\dim(S^{(4,4)}(\mathbb{C}^3)_{\alpha}) & 0 & 0 & 1 & 1 & 1
\end{array}\] 
In the above table the sum of the dimensions in the column of $\alpha$ agrees with 
$|\mathcal{N}_{\alpha}|$, hence we showed that  
\[\mathcal{K}_{\le 8}=
\langle \mathcal{R}(4)_{6,2}\rangle_{GL_3(\mathbb{C})}
\oplus \langle \mathcal{R}(4)_{4,4}\rangle_{GL_3(\mathbb{C})}
\oplus \mathcal{I}_{\le 8}.\] 
Finally we turn to the relations of degree $10$. 
Our study of the degree $8$ relations implies that any product 
$\pi_{i_1,i_2,i_3}\pi_{j_1,j_2,j_3}$ 
is congruent modulo $\mathcal{K}+\mathcal{H}$ to a linear combination of 
elements from the $S_3$-orbit of $\pi_{2,1,1}\pi_{1,2,1}$ and 
products of the variables of $\mathcal{F}(4,3)$ 
involving at most one factor of the form $\pi_{k_1,k_2,k_3}$. 
Moreover, 
\[\pi_{2,1,1}\pi_{1,2,1}\rho_{i,j,k}\equiv_{\mathcal{I}} 0\] 
by \eqref{eq:p_211q_110} and \eqref{eq:p_211q_011}. 
Taking into account \eqref{eq:ABC} we conclude that 
\begin{align} \label{eq:F(3)_10} 
\mathcal{F}(4,3)_{10} \text{ is spanned modulo }\mathcal{H}+\mathcal{K}  \text{ by products involving at most one variable} 
\\ \notag  \text{ of the form }\pi_{i_1,i_2,i_3} 
\text{ and not divisible by any element of }A\cup B\cup C.
\end{align} 
For a decreasing $\alpha$ with $\alpha_1+\alpha_2+\alpha_3=10$ denote by 
$\mathcal{V}_{\alpha}$ the subset of $\mathcal{M}_{\alpha}\setminus \mathcal{T}_{\alpha}$ 
obtained by removing  the products divisible by some $\pi_{i_1,i_2,i_3}\pi_{j_1,j_2,j_3}$ or 
by any element of $A\cup B\cup C$. By \eqref{eq:F(3)_10} it is sufficient to prove that 
every element of $\mathcal{V}_{\alpha}$ is congruent modulo 
$\mathcal{K}+\mathcal{H}$ to an element in $\mathrm{Span}_{\mathbb{C}}\mathcal{T}_{\alpha}$. 
The table below gives $\mathcal{T}_{\alpha}$ and $\mathcal{V}_{\alpha}$ for all decreasing $\alpha$ with $\alpha_3\neq 0$ and $\sum\alpha_i=10$.

\[\begin{array}{c||c|c} 
\alpha &  \mathcal{T}_{\alpha} &\mathcal{V}_{\alpha}  
\\ \hline \hline 
(8,1,1) & & \\ \hline
(7,2,1) & & \\ \hline 
(6,3,1) & & 
\\ \hline 
(6,2,2) &  & 
\\ \hline 
(5,4,1) & & \rho_{1,1,0}^4\rho_{1,0,1}  
\\ & & 
\\ \hline 
(5,3,2) & & \rho_{1,1,0}^3\rho_{1,0,1}^2
\\ \hline 
(4,4,2) & &  
\\ \hline 
(4,3,3) & \pi_{3,1,0}\rho_{1,0,1}\rho_{0,1,1}^2 & \pi_{1,3,0}\rho_{1,0,1}^3, 
\ \pi_{1,0,3}\rho_{1,1,0}^3
\\ \hline
\end{array}\]
We see that $\mathcal{V}_{\alpha}$ is non-empty only for 
$\alpha\in\{(5,4,1),\ (5,3,2),\ (4,3,3)\}$. 

$\mathbf{(5,4,1)}$: The relation $\mathcal{R}_{2n-4,4}$ in the special case $n=4$ is 
\begin{equation*}%\label{eq:R_{4,4}^{4,4,0}} 
\mathcal{R}(4)_{4,4}:=\pi_{4,0,0}\pi_{0,4,0}-4\pi_{3,1,0}\pi_{1,3,0}+3\pi_{2,2,0}^2-16(\rho_{1,1,0}^2-\rho_{2,0,0}\rho_{0,2,0})^2. 
\end{equation*} 
Apply $\frac 14E_{3,2}$ to $\mathcal{R}(4)_{4,4}$ we  get 
\begin{align*} 
\mathcal{R}(4)_{4,4}^{4,3,1}:=\frac 14E_{3,2}.\mathcal{R}(4)_{4,4}=
\pi_{4,0,0}\pi_{0,3,1}-\pi_{3,0,1}\pi_{1,3,0}-3\pi_{3,1,0}\pi_{1,2,1}
+3\pi_{2,2,0}\pi_{2,1,1}
\\ -16(\rho_{2,0,0}\rho_{0,2,0}-\rho_{1,1,0}^2)(\rho_{2,0,0}\rho_{0,1,1}-\rho_{1,1,0}\rho_{1,0,1}). 
\end{align*} 
We deduce from $\mathcal{R}(4)_{4,4}^{4,3,1}\in \mathcal{K}$ that 
\begin{equation}\label{eq:q_110^3q_101} 
\rho_{1,1,0}^3\rho_{1,0,1}\equiv  \frac{1}{16}(-\pi_{3,0,1}\pi_{1,3,0}-3\pi_{3,1,0}\pi_{1,2,1}+3\pi_{2,2,0}\pi_{2,1,1}). 
\end{equation}
Multiplying \eqref{eq:q_110^3q_101} by $\rho_{1,1,0}$ and using 
\[\pi_{1,3,0}\rho_{1,1,0}\equiv 0,\ \pi_{3,1,0}\rho_{1,1,0}\equiv 0,\ 
\pi_{2,2,0}\rho_{1,1,0}\equiv 0\] 
we get 
\[\rho_{1,1,0}^4\rho_{1,0,1} \equiv_{\mathcal{I}} 0.\] 

$\mathbf{(5,3,2)}$: 
Multiplying \eqref{eq:q_110^3q_101} by $\rho_{1,0,1}$ and using 
\[\pi_{3,0,1}\rho_{1,0,1}\equiv_{\mathcal{I}} 0,\ \pi_{1,2,1}\rho_{1,0,1}\equiv_{\mathcal{I}} 0,\ 
\pi_{2,1,1}\pi_{1,0,1}\equiv_{\mathcal{I}} 0\] 
we get
\[\rho_{1,1,0}^3\rho_{1,0,1}^2\equiv_{\mathcal{I}} 0.\] 

$\mathbf{(4,3,3)}$: 
Applying the transposition $(1,2)\in S_3$ to \eqref{eq:p_220q_101} we get 
\[\pi_{1,3,0}\rho_{1,0,1}\equiv_{\mathcal{I}}\pi_{2,2,0}\rho_{0,1,1}.\] 
Multiplying this by $\rho_{1,0,1}^2$ we get 
\[\pi_{1,3,0}\rho_{1,0,1}^3\equiv_{\mathcal{I}}\pi_{2,2,0}\rho_{1,0,1}^2\rho_{0,1,1}
\equiv_{\mathcal{I}}\pi_{3,1,0}\rho_{1,0,1}\rho_{0,1,1}^2\] 
(for the second congruence see \eqref{eq:p_220q_101}). 
Finally, applying the transposition $(2,3)\in S_3$ to the congruence 
$\pi_{1,3,0}\rho_{1,0,1}^3\equiv_{\mathcal{I}} \pi_{3,1,0}\rho_{1,0,1}\rho_{0,1,1}^2$ 
and using \eqref{eq:p_301q_110} we obtain 
\[\pi_{1,0,3}\rho_{1,1,0}^3\equiv_{\mathcal{I}} -\pi_{3,1,0}\rho_{1,0,1}\rho_{0,1,1}^2.\]
\end{proofof} 

\begin{theorem}\label{thm:n=4} 
For arbitrary $m\ge 3$ the kernel of $\varphi(4,m):\mathcal{F}(4,m)\to \mathbb{C}[V^m]^{D_8}$ is minimally generated as a 
$GL$-ideal by 
$\mathcal{R}_{2,2,2}$, $\mathcal{R}(4)_{4,2}$, $\mathcal{R}(4)_{6,2}$, $\mathcal{R}(4)_{4,4}$. 
\end{theorem} 

\begin{proof}
This is an immediate consequence of Proposition~\ref{prop:n=4, m=3} and 
Theorem~\ref{thm:m=3 sufficient}. 
\end{proof} 

The symmetric group $S_3$ is isomorphic to the dihedral group $D_6$ of order $6$, 
and the natural $3$-dimensional permutation representation of $S_3$ can be identified with the sum of the trivial representation and the defining $2$-dimensional representation of $D_6$. Therefore \cite[Theorem 3.1]{domokos-puskas} dealing with multisymmetric polynomials can be restated in the notation of the present paper as follows: 

\begin{theorem}\label{thm:n=3} 
For arbitrary $m\ge 3$ the kernel of $\varphi(3,m):\mathcal{F}(3,m)\to \mathbb{C}[V^m]^{D_6}$ is minimally generated as a 
$GL$-ideal by 
$\mathcal{R}_{2,2,2}$, $\mathcal{R}(3)_{3,2}$, $\mathcal{R}(3)_{4,2}$. 
\end{theorem} 

\begin{remark} \begin{itemize} \item[(i)] The method of the present paper would yield 
a shorter proof of Theorem~\ref{thm:n=3} than the proof in \cite{domokos-puskas}. 

\item[(ii)]For $n=2$ the group  $D_4$ is isomorphic to Klein's four group, so it is abelian and its action on $V$ is diagonalizable. 
In an appropriate basis we have  $\mathbb{C}[V]^{D_4}=\mathbb{C}[x^2,y^2]$, 
and $\mathbb{C}[V^m]=\mathbb{C}[x_ix_j,y_iy_j\mid 1\le 1\le j\le m]$. 
The corresponding $GL$-ideal of relations among the generators 
of $\mathbb{C}[V^m]^{D_4}$ is generated by the two relations 
$x_1^2x_2^2-(x_1x_2)^2=0$ and  $y_1^2y_2^2-(y_1y_2)^2=0$. 
For some results on presentations of rings of invariants of abelian groups see 
for example \cite{domokos_graz} and the references therein. 
\end{itemize}
\end{remark} 

%%%%%%%%%%%%%%%%%%%%%%%

\section{Some computations} 

Next we determine the $GL_m(\mathbb{C})$-module structure of $\mathbb{C}[V^m]^{D_{2n}}$.  It turns out that the multiplicities of the irreducible summands 
are conveniently expressed in terms of the coefficients of the Hilbert series 
of $\mathbb{C}[V^m]^{D_{2n}}$. 
Denote by $h(d)$ the dimension of the degree $d$ homogeneous component of $\mathbb{C}[x,y]^{D_{2n}}$. Note that in the formal power series ring $\mathbb{Z}[[t]]$ we have the equality 
\begin{equation}
\label{eq:Hilbert series}
\sum_{d=0}^{\infty}h(d)t^d=\frac{1}{(1-t^2)(1-t^n)}.
\end{equation}  

\begin{proposition}\label{prop:GL-structure of invariants} 
The multiplicity of $S^{\lambda}(\mathbb{C}^m)$ as a summand in $\mathbb{C}[V^m]^{D_{2n}}$ is 
non-zero only if $\mathrm{ht}(\lambda)\le 2$, and in this case the multiplicity is 
\[\begin{cases} h(\lambda_1-\lambda_2) &\text{ if }2\mid \lambda_2 \\
h(\lambda_1-\lambda_2-n) &\text{ if }2\nmid \lambda_2.\end{cases} \]
\end{proposition} 

\begin{proof} 
Identify $V^m$ with the space $\mathbb{C}^{2\times m}$ of $2\times m$ matrices, endowed with the 
$GL_2(\mathbb{C})\times GL_m(\mathbb{C})$-action $(g,h)\cdot A:=gAh^{-1}$ for $g\in GL_2(\mathbb{C})$, 
$h\in GL_m(\mathbb{C})$, and $A\in \mathbb{C}^{2\times m}$. This induces an action of $GL_2(\mathbb{C})\times GL_m(\mathbb{C})$ on $\mathbb{C}[V^m]$ via $\mathbb{C}$-algebra automorphisms in the standard way: 
$((g,h)\cdot f)(A)=f(g^{-1}Ah)$ for $g\in GL_2(\mathbb{C})$, 
$h\in GL_m(\mathbb{C})$, $A\in \mathbb{C}^{2\times m}$ and $f\in \mathbb{C}[V^m]$. 
Note that the restriction of this action to the subgroup $GL_m(\mathbb{C})$ 
agrees with the $GL_m(\mathbb{C})$-action on $\mathbb{C}[V^m]$ considered in the previous sections.   
By the Cauchy Formula (see for example \cite[Section 9.6.3]{procesi}) we have 
\[\mathbb{C}[V^m]_d\cong \sum_{\lambda\in\mathrm{Par}_m(d),\ \mathrm{ht}(\lambda)\le 2}S^{\lambda}(V^*)\otimes S^{\lambda}(\mathbb{C}^m)\] 
as $GL_2(\mathbb{C})\times GL_m(\mathbb{C})$-modules. 
It follows that we have the following isomorphism of $GL_m(\mathbb{C})$-modules: 
\[\mathbb{C}[V^m]^{D_{2n}}_d\cong 
\sum_{\lambda\in\mathrm{Par}_m(d),\ \mathrm{ht}(\lambda)\le 2}
S^{\lambda}(V^*)^{D_{2n}}\otimes  S^{\lambda}(\mathbb{C}^m).\]
Note that for a partition $\lambda=(\lambda_1,\lambda_2)$,  we have the 
$D_{2n}$-module isomorphism 
$S^{\lambda}(V^*)\cong S^{\lambda_1-\lambda_2}(V^*)$ when $\lambda_2$ is even, 
whereas for $\lambda_2$ odd we have 
$S^{\lambda}(V^*)\cong \det\otimes S^{\lambda_1-\lambda_2}(V^*)$, where 
$\det$ stands for the $1$-dimensional representation of $D_{2n}$ obtained by composing the determinant with the $2$-dimensional defining representation of $D_{2n}$. 
Now the result follows from the well-known description of $\mathbb{C}[V]\cong S(V^*)$ as a $D_{2n}$-module. 
\end{proof} 

For a fixed $n$ 
Proposition~\ref{prop:GL-structure of invariants} and Corollary~\ref{cor:GL-structure of F} allow us to compute the multiplicities of the simple $GL_3(\mathbb{C})$-module summands of the homogeneous components of $\ker(\bar\varphi(n ,3))$ up to degree $2n+2$. 
We shall do this computation for the case $n=4$. 
To simplify notation write $S(\lambda)$ for $S^{\lambda}(\mathbb{C}^m)$. 
Corollary~\ref{cor:GL-structure of F} yields 
\begin{align} \label{eq:F10}
(\bar{\mathcal{D}}(m)\otimes\mathcal{E}(4,m))_{\le 10}&\cong 
S(0)+S(2)+2S(4)+S(2,2)+2S(6)+S(5,1)
\\ \notag &+2S(4,2)+3S(8)+S(7,1)+4S(6,2)+S(5,3)+3S(4,4)
\\ \notag &+3S(10)+2S(9,1)+5S(8,2)+3S(7,3)+5S(6,4)
\\ \notag  &+S(5,2,1)+S(4,2,2)+2S(7,2,1)
 \\ \notag &+2S(6,3,1)+2S(5,4,1)+2S(6,2,2)
  \\ \notag &+S(5,3,2)+2S(4,4,2).
\end{align}

By Proposition~\ref{prop:GL-structure of invariants} and \eqref{eq:Hilbert series} we have 
\begin{align}\label{eq:GL-structure of D4-invariants} 
\mathbb{C}[V^m]^{D_{8}}_{\le 10}
& \cong S(0)+S(2)+2S(4)+S(2,2)+2S(6)+S(5,1)
\\ \notag &+S(4,2)+3S(8)+S(7,1)+2S(6,2)+S(4,4)+3S(10)
\\ \notag &+2S(9,1) +2S(8,2)+S(7,3)+S(6,4).  
\end{align}

Combining \eqref{eq:F10} and \eqref{eq:GL-structure of D4-invariants}  
we get 
\begin{align}\label{eq:GL-structure of ideal of D4-relations} 
\ker(\bar\varphi(4,m)_{\le 10})
&\cong S(4,2)+2S(6,2)+2S(4,4)+3S(8,2)+4S(6,4)
\\ \notag &+S(5,3)+2S(7,3)+S(5,2,1)+S(4,2,2)+2S(7,2,1)
\\ \notag &+2S(6,3,1)+2S(6,2,2)+2S(5,4,1)+2S(4,4,2)+S(5,3,2).
\end{align} 

An alternative approach to Theorem~\ref{thm:n=4} would be to find the highest weight vectors in the $GL$-ideal generated by $\mathcal{R}_{2,2,2}$, 
$\mathcal{R}(4)_{4,2}$, $\mathcal{R}(4)_{6,2}$, $\mathcal{R}(4)_{4,4}$ 
required by the decomposition \eqref{eq:GL-structure of ideal of D4-relations}. 
However, this seems to be more laborious than the approach in Section~\ref{sec:n=4} based on the Hironaka decomposition.

%%%%%%%%%%%%%%%%%%%%%%%%%%%%%%%%%%

\end{document}